\documentclass[oneside,reqno]{amsart}

\usepackage[utf8]{inputenc} 

\usepackage[top=1.0 in, left=1.0 in, bottom=1.0 in, right=1.0 in]{geometry}
\usepackage{graphicx}

\usepackage{mathrsfs}
\usepackage{manfnt}

\usepackage{lipsum}

\usepackage{standalone}
\usepackage{etoolbox}
\newtoggle{tufte}
\toggletrue{tufte}
\usepackage{varwidth}
\usepackage{emptypage}

\usepackage{amsmath}
\usepackage{accents}
\usepackage{amsthm}
\usepackage{bbm}
\usepackage{thmtools}
\usepackage{amssymb}
\usepackage{latexsym}
\usepackage{geometry}
\usepackage{setspace}
\usepackage{manfnt}
\usepackage{mathdots}
\usepackage{cancel}
\usepackage{indentfirst}

\usepackage{algpseudocode}
\usepackage{algorithm}

\usepackage{ragged2e}
\usepackage{enumerate}
\usepackage{listings}
\usepackage{tabularx}
\usepackage{xparse}
\usepackage{xspace}
\usepackage[shortlabels]{enumitem}
\usepackage{multirow}
\usepackage{todonotes}
\usepackage{dcolumn}
\usepackage{array}
\usepackage{colortbl, xcolor}
\usepackage{longtable}
\usepackage{tabu}
\usepackage{python}

\usepackage{asymptote}

\usepackage{ytableau}

\usepackage{blkarray}

\allowdisplaybreaks

\usepackage{graphicx}
\setkeys{Gin}{width=\linewidth,totalheight=\textheight,keepaspectratio}
\graphicspath{{Graphics/}}

\usepackage{hyperref}
\hypersetup{colorlinks=false}

\usepackage{fancyvrb}
\fvset{fontsize=\normalsize}

\usepackage{xspace}

\usepackage{units}

\declaretheorem[numberwithin=section,name=Theorem]{theorem}
\declaretheorem[name=Lemma,sibling=theorem]{lemma}
\declaretheorem[name=Corollary,sibling=theorem]{corollary}
\declaretheorem[name=Definition,sibling=theorem,style=definition]{definition}
\declaretheorem[name=Proposition,sibling=theorem]{proposition}

\declaretheorem[name=Open Problem]{openprob}

\declaretheorem[name=Example,sibling=theorem,style=definition]{example}

\declaretheorem[name=Note,style=remark]{note}
\declaretheorem[name=Remark,sibling=note,style=remark]{remark}

\newenvironment{notation}[1][Notation]{\begin{trivlist}
\item[\hskip \labelsep {\bfseries #1}.]}
{\end{trivlist}}

\setlist[enumerate]{label=(\arabic*)}

\newlist{exercises}{enumerate}{1}
\setlist[exercises]{label=Ex. \arabic*}

\def \C {\mathbb{C}}
\newcommand{\Z}{\mathbb{Z}}

\newcommand{\vl}[1]{\overline{#1}}
\newcommand{\dl}{\displaystyle}
\newcommand{\thus}{{.\raise 4pt\hbox{.}.\;}}

 2

\newcommand{\Anlg}[1]{\operatorname{\mathscr{A}}(#1)}
\newcommand{\tAnlg}[1]{\widetilde{\operatorname{\mathscr{A}}}(#1)}

\newcommand{\Hom}[2]{\operatorname{Hom}_{#1}(#2)}

\newcommand{\Res}[2]{\operatorname{Res}_{#1}^{#2}}

\DeclareMathOperator{\supp}{supp}

\DeclareMathOperator{\gl}{\mathfrak{gl}}
\DeclareMathOperator{\Id}{Id}
\DeclareMathOperator{\Aut}{Aut}

\DeclareMathOperator{\Frac}{Frac}

\DeclareMathOperator{\fdMod}{-\underline{Mod}^{f.d.}}
\DeclareMathOperator{\Span}{Span}
\DeclareMathOperator{\Gal}{Gal}

\begin{document}

\title{An Extension of $U(\gl_n)$ Related to the Alternating Group and Galois Orders}
\author{Erich C. Jauch}
\date{}

\address{Department of Mathematics, Iowa State University, Ames, IA-50011, USA}
\email{ecjauch@iastate.edu}
\urladdr{ecjauch.com}

\maketitle


\begin{abstract}
In 2010, V. Futorny and S. Ovsienko gave a realization of $U(\gl_n)$ as a subalgebra
of the ring of invariants of a certain noncommutative ring with respect to the action
of $S_1\times S_2\times\cdots\times S_n$, where $S_j$ is the symmetric group on $j$
variables. An interesting question is what a similar algebra would be in the invariant
ring with respect to a product of alternating groups. In this paper we define such an
algebra, denoted $\Anlg{\gl_n}$, and show that it is a Galois ring. For $n=2$, we show that it is a generalized Weyl algebra, and for $n=3$ provide generators and a list of verified relations. We also discuss some
techniques to construct Galois orders from Galois rings. Additionally, we study categories of
finite-dimensional modules and generic Gelfand-Tsetlin modules over $\Anlg{\gl_n}$.
Finally, we discuss connections between the Gelfand-Kirillov Conjecture, $\Anlg{\gl_n}$, and the positive solution to Noether's problem for the alternating group.
\end{abstract}

{\bf Keywords:} alternating group, enveloping algebra, Gelfand-Kirillov Conjecture, Gelfand-Tsetlin modules, weight modules.

{\bf 2010 Mathematics Subject Classification:} 16G99, 17B10

\section{Introduction}\label{sec: introduction}

The study of algebra-subalgebra pairs is an important technique used in the representation theory
of Lie algebras \cite{LM73},\cite{DFO94}. In 2010, Futorny and Ovsienko focused on so called semicommutative
pairs $\Gamma\subset\mathscr{U}$, where $\mathscr{U}$ is an associative (noncommutative)
$\C$-algebra and $\Gamma$ is an integral domain \cite{FO10}.
This situation generalizes the pair $(\Gamma,U(\gl_n))$ where $\Gamma$
is the \emph{Gelfand-Tsetlin} subalgebra $\Gamma=\C\langle\cup_{k=1}^nZ(U(\gl_k))\rangle$
\cite{Gelfand:1950ihs}, \cite{DFO94}. \emph{Galois rings} and \emph{Galois orders} were originally
defined and studied by Futorny and Ovsienko in \cite{FO10} and \cite{FO14}. They form a
collection of algebras that contains many important examples including:
\emph{generalized Weyl algebras} defined by independently by Bavula \cite{BavulaGWA}
and Rosenberg \cite{rosenberg_1995} in the early nineties, the universal enveloping algebra
of $\gl_n$, shifted Yangians and finite $W$-algebras \cite{FMO10}, Coulomb branches
\cite{Webster19}, and $U_q(\gl_n)$ \cite{FH14}. Their structures and representations
have been studied in \cite{EFG18}, \cite{QuanutmLinearGaloisAlgebras},
\cite{Hartwig20}, and \cite{MV18}.

In \cite{FO10}, Futorny and Ovsienko described $U(\gl_n)$ as the subalgebra of
the ring of invariants of a certain noncommutative ring with respect to the action
of $S_1\times S_2\times\cdots\times S_n$, where $S_j$ is the symmetric group on $j$
variables such that $U(\gl_n)$ was a Galois order with respect to its
Gelfand-Tsetlin subalgebra $\Gamma$.

We recall in Galois theory, given a Galois extension $L/K$ with $\Gal(L/K)=G$ the
subgroups $\widetilde{G}$ of $G$ correspond to intermediate fields $\widetilde{K}$
with $\Gal(L/\widetilde{K})=\widetilde{G}$ with normal subgroups of particular
interest. Since $S_n$ has only one normal subgroup for $n\geq5$, one might wonder
what the object similar to $U(\gl_n)$ would be if we considered the invariants with
respect to the normal subgroup $A_1\times A_2\times\cdots\times A_n$, where $A_j$
is the alternating group on $j$ variables. This paper describes such an algebra,
denoted by $\Anlg{\gl_n}$ (see Definition \ref{def: A(gln)}). This provides the
first natural example of a Galois ring whose ring $\Gamma$ is not a semi-Laurent
polynomial ring, that is, a tensor product of polynomial rings and Laurent
polynomial rings. Additionally, our symmetry group $A_1\times A_2\times\cdots\times A_n$ is not
a complex reflection group. Our algebra $\Anlg{\gl_n}$ is an extension of $U(\gl_n)$
by $n-1$ elements $\mathcal{V}_2,\ldots,\mathcal{V}_n$.
In Proposition \ref{prop: basic properties of A(gln)}, we prove some
properties of $\Anlg{\gl_n}$ that are quite similar to $U(\gl_n)$.
For example, it is shown that the ``Weyl Group'' of $\Anlg{\gl_n}$ is the
alternating group $A_n$, in the sense that there is a natural extension
$\widetilde{\varphi}_{\rm HC}$ of the Harish-Chandra homomorphism 
$\varphi_{\rm HC}\colon Z(U(\gl_n))\rightarrow S(\mathfrak{h})\cong\C[x_1,\ldots,x_n]$,
such that
\[
\widetilde{\varphi}_{\rm HC}\colon Z(\Anlg{\gl_n})\xrightarrow{\cong}\C[x_1,\ldots,x_n]^{A_n}.
\]
Moreover, there is a chain of subalgebras
$\Anlg{\gl_1}\subset\Anlg{\gl_2}\subset\cdots\subset\Anlg{\gl_n}$.
In Section \ref{sec: n=2}, we give multiple descriptions of $\Anlg{\gl_2}$ and
prove it is realizable as a Galois order. Example \ref{ex: Gammatilde is not maximal comm}
shows that $\Anlg{\gl_n}$ is not a Galois order for $n\geq3$. The rest
of Section \ref{sec: n=3} lists a set of generators and some verified relations for
$\Anlg{\gl_3}$, but this list may be incomplete. In Section \ref{sec: fd modules over A(gln)}, 
we show that the category of finite-dimensional modules in not semi-simple and
classify simple finite-dimensional weight modules. In Section 
\ref{sec: locolization theorem}, we provide a technique to turn a general Galois ring
into a Galois order that is related to localization (see Theorem
\ref{thm: Order if centralizer is a nice localization for arbitrary G}).
We use this to prove that a family of simple examples are Galois orders (see Corollary
\ref{cor: toy example is an order}) and that a localization of $\Anlg{\gl_n}$ is a
(co-)principal Galois order over the localized $\widetilde{\Gamma}$ (see Definition
\ref{def: principal and co-principal Galois orders} and Corollary
\ref{cor: Localized A(gln) is an order}). We use this localization
to construct canonical Gelfand-Tsetlin modules over $\Anlg{\gl_n}$ in
Section \ref{sec: GT Modules over A(gln)}. Finally, in Section
\ref{sec: GK conj for A(gln)}, we compute the division ring of fractions and
prove, that for $n\leq5$, $\Anlg{\gl_n}$ satisfies the Gelfand-Kirillov conjecture
(see \cite{GK66}). For the latter, we use Maeda's positive
solution to Noether's problem for the alternating group $A_5$ \cite{MAEDA1989418},
and Futorny-Schwarz's Theorem 1.1 in \cite{FSNoncommutativeNoetherVSClassicalNoether}.

\subsection{Galois orders}\label{sec: Galois Orders}
Galois orders were introduced in \cite{FO10}. We will be following
the set up from \cite{Hartwig20}. Let $\Lambda$ be an integrally closed
domain, $G$ a finite subgroup of $\Aut(\Lambda)$, and $\mathscr{M}$ a
submonoid of $\Aut(\Lambda)$. We will adhere to the following
assumptions for the entire paper:
\begin{align*}
{\rm (A1)}\quad & (\mathscr{M}\mathscr{M}^{-1})\cap G=1_{\Aut{}{\Lambda}} & \text{(\emph{separation})}\\
{\rm (A2)}\quad & \forall g\in G, \forall\mu\in\mathscr{M}\colon {}^g\mu=g\circ\mu\circ g^{-1}\in\mathscr{M} &
\text{(\emph{invariance})}\\
{\rm (A3)}\quad & \Lambda \text{ is Noetherian as a module over } \Lambda^G & \text{(\emph{finiteness})}
\end{align*}

Let $L=\Frac(\Lambda)$ and $\mathscr{L}=L\#\mathscr{M}$, the skew monoid ring, which is defined as the
free left $L$-module on $\mathscr{M}$ with multiplication given by
$a_1\mu_1\cdot a_2\mu_2=(a_1\mu_1(a_2))(\mu_1\mu_2)$ for $a_i\in L$ and $\mu_i\in\mathscr{M}$.
As $G$ acts on $\Lambda$ by automorphisms, we can easily extend this action to $L$, and by {\rm (A2)},
$G$ acts on $\mathscr{L}$. So we consider the following $G$-invariant subrings
$\Gamma=\Lambda^G$, $K=L^G$, and $\mathscr{K}=\mathscr{L}^G$.

A benefit of these assumptions is the following lemma.

\begin{lemma}[\cite{Hartwig20}, Lemma 2.1 (ii), (iv) \& (v)]\label{lem: Hartwig big lemma}
\item
\begin{enumerate}[\rm (i)]
\item $K=\Frac(\Gamma)$.
\item $\Lambda$ is the integral closure of $\Gamma$ in $L$.
\item $\Lambda$ is a finitely generated $\Gamma$-module and a Noetherian ring.
\end{enumerate}
\end{lemma}

What follows are some definitions and propositions from \cite{FO10}.

\begin{definition}[\cite{FO10}]
A finitely generated $\Gamma$-subring $\mathscr{U}\subseteq\mathscr{K}$ is called a \emph{Galois $\Gamma$-ring}
(or \emph{Galois ring with respect to $\Gamma$}) if $K\mathscr{U}=\mathscr{U}K=\mathscr{K}$.
\end{definition}

\begin{definition}
Let $u\in\mathscr{L}$ such that $u=\sum_{\mu\in\mathscr{M}}a_\mu \mu$. The \emph{support of $u$ over $\mathscr{M}$} is the following:
\[
\supp u=\bigg\{\mu\in\mathscr{M}~\Big\vert~a_\mu\neq0\text{ for }u=\sum_{\mu\in\mathscr{M}}a_\mu \mu\bigg\}
\]
\end{definition}

\begin{proposition}[\cite{FO10}, Proposition 4.1]\label{prop: Gamma Ring Alt Conditions}
Assume a $\Gamma$-ring $\mathscr{U}\subseteq\mathscr{K}$ is generated by $u_1,\ldots,u_k\in \mathscr{U}$.
\begin{enumerate}[\rm (1)]
\item If $\bigcup_{i=1}^k\supp u_i$ generate $\mathscr{M}$ as a monoid, then $\mathscr{U}$ is a Galois ring.
\item If $L\mathscr{U}=L\#\mathscr{M}$, then $\mathscr{U}$ is a Galois ring.
\end{enumerate}
\end{proposition}

\begin{theorem}[\cite{FO10},  Theorem 4.1 (4)]\label{thm: Center of a Galois Ring}
Let $\mathscr{U}$ be a Galois $\Gamma$-ring. Then the center $Z(\mathscr{U})$ of the algebra
$\mathscr{U}$ equals $\mathscr{U}\cap K^{\mathscr{M}}$,
where $K^{\mathscr{M}}=\{k\in K\mid \mu(k)=k~\forall\mu\in\mathscr{M}\}$
\end{theorem}

\begin{definition}[\cite{FO10}]\label{def: Galois Order defintion}
A Galois $\Gamma$-ring $\mathscr{U}$ in $\mathscr{K}$ is a \emph{left} (respectively \emph{right})
\emph{Galois $\Gamma$-order in $\mathscr{K}$} if for any finite-dimensional left (respectively right)
$K$-subspace $W\subseteq\mathscr{K}$, $W\cap\mathscr{U}$ is a finitely generated left (respectively
right) $\Gamma$-module. A Galois $\Gamma$-ring $\mathscr{U}$ in $\mathscr{K}$ is a \emph{Galois
$\Gamma$-order in $\mathscr{K}$} if $\mathscr{U}$ is a left and right Galois $\Gamma$-order in
$\mathscr{K}$.
\end{definition}

\begin{definition}[\cite{DFO94}]\label{def: Harish-Chandra subalg}
Let $\Gamma\subset\mathscr{U}$ be a commutative subalgebra. $\Gamma$ is called a
\emph{Harish-Chandra subalgebra} in $\mathscr{U}$ if for any $u\in\mathscr{U}$, $\Gamma u\Gamma$
is finitely generated as both a left and as a right $\Gamma$-module.
\end{definition}

Let $\mathscr{U}$ be a Galois ring and $e\in\mathscr{M}$ the unit element. We denote
$\mathscr{U}_e=\mathscr{U}\cap Le$.

\begin{theorem}[\cite{FO10}, Theorem 5.2]\label{thm: Galois Order condition}
Assume that $\mathscr{U}$ is a Galois ring, $\Gamma$ is finitely generated and $\mathscr{M}$
is a group.
\begin{enumerate}[\rm (1)]
\item Let $m\in\mathscr{M}$. Assume $m^{-1}(\Gamma)\subseteq\Lambda$ (respectively $m(\Gamma)\subseteq\Lambda$).
Then $\mathscr{U}$ is right (respectively left) Galois order if and only if $\mathscr{U}_e$ is an
integral extension of $\Gamma$.
\item Assume that $\Gamma$ is a Harish-Chandra subalgebra in $\mathscr{U}$. Then $\mathscr{U}$ is
a Galois order if and only if $\mathscr{U}_e$ is an integral extension of $\Gamma$.
\end{enumerate}
\end{theorem}

The following are some useful results from \cite{Hartwig20}.

\begin{proposition}[\cite{Hartwig20}, Proposition 2.14]\label{prop: Gamma maxl comm in a Galois Order}
$\Gamma$ is maximal commutative in any left or right Galois $\Gamma$-order $\mathscr{U}$ in $\mathscr{K}$.
\end{proposition}

\begin{lemma}[\cite{Hartwig20}, Lemma 2.16]\label{lem: Order Containment Implication}
Let $\mathscr{U}_1$ and $\mathscr{U}_2$ be two Galois $\Gamma$-rings in
$\mathscr{K}$ such that $\mathscr{U}_1\subseteq\mathscr{U}_2$. If
$\mathscr{U}_2$ is a Galois $\Gamma$-order, then so too is $\mathscr{U}_1$.
\end{lemma}

It is common to write elements of $L$ on the right side of elements of $\mathscr{M}$.

\begin{definition}
For $X=\sum_{\mu\in\mathscr{M}}\mu\alpha_\mu\in\mathscr{L}$ and $a\in L$
defines the \emph{evaluation of $X$ at $a$} to be
\[
X(a)=\sum_{\mu\in\mathscr{M}}\mu(\alpha_\mu\cdot a)\in L.
\]
Similarly defined is \emph{co-evaluation} by
\[
X^\dagger(a)=\sum_{\mu\in\mathscr{M}}\alpha_\mu\cdot(\mu^{-1}(a))\in L
\]
\end{definition}

The following was independently defined by \cite{Vishnyakova17} called the \emph{universal ring}. 

\begin{definition}
The \emph{standard Galois $\Gamma$-order} is as follows:
\[
\mathscr{K}_\Gamma:=\{X\in\mathscr{K}\mid X(\gamma)\in\Gamma~\forall\gamma\in\Gamma\}.
\]
Similarly we define the \emph{co-standard Galois $\Gamma$-order} by
\[
{}_\Gamma{\mathscr{K}}:=\{X\in\mathscr{K}\mid X^\dagger(\gamma)\in\Gamma
~\forall\gamma\in\Gamma\}.
\]
\end{definition}

\begin{definition}\label{def: principal and co-principal Galois orders}
Let $\mathscr{U}$ be a Galois $\Gamma$-ring in $\mathscr{K}$. If $\mathscr{U}\subseteq\mathscr{K}_\Gamma$ (resp.
$\mathscr{U}\subseteq{}_\Gamma{\mathscr{K}}$), then $\mathscr{U}$ is called a \emph{principal} (resp.
\emph{co-principal}) \emph{Galois $\Gamma$-order}.
\end{definition}

In \cite{Hartwig20} it was shown that any (co-)principal Galois $\Gamma$-order is a
Galois order in the sense of Definition \ref{def: Galois Order defintion}.


\section{Defining the Extension}\label{sec: def Agln}

\subsection{Galois order realization of \texorpdfstring{$U(\gl_n)$}{U(gln)}}\label{sec: U(gln) Galois order realization}
We first recall the realization of $U(\gl_n)$ as a Galois $\Gamma$-order from \cite{FO10}.
Let $\Lambda=\C[x_{ki}\mid 1\leq i\leq k\leq n]$ the polynomial algebra
in indeterminates $x_{ki}$,
$\mathbb{S}_n=S_1\times S_2\times\cdots\times S_n$, and
$\Gamma=\Lambda^{\dl\mathbb{S}_n}=\C[e_{ki}\mid1\leq i\leq k\leq n]$.
Here
\begin{equation}
e_{ki}=e_{ki}(x_{k1},\ldots,x_{kk})=\sum_{1\leq j_1<\cdots<j_i\leq k}x_{kj_1}\cdots x_{kj_i}    
\end{equation}
are the elementary symmetric polynomials. Also, let $L=\Frac(\Lambda)$
and $K=\Frac(\Gamma)$. Now, we construct a skew monoid ring. Let
$\mathscr{M}$ be the subgroup of $\Aut(\Lambda)$ generated by $\{\delta^{ki}\}_{1\leq i\leq k\leq n-1}$,
where $\delta^{ki}$ is defined by
\begin{equation}
\delta^{ki}(x_{\ell j})=x_{\ell j}-\delta_{\ell k}\delta_{ij}.
\end{equation}
We observe that $\mathscr{M}\cong\Z^{n(n-1)/2}$. Let $\mathscr{L}=L\#\mathscr{M}$ and
$\mathscr{K}=(L\#\mathscr{M})^{\dl\mathbb{S}_n}$. In \cite{FO10}, the authors describe
an embedding $\varphi\colon U(\gl_n)\rightarrow\mathscr{K}$ defined by
\begin{equation}\label{eq: Ugln map}
\varphi(E_{k}^\pm)=\sum_{i=1}^k(\delta^{ki})^{\pm1}a_{ki}^\pm,\qquad
\varphi(E_{kk})=\sum_{j=1}^k(x_{kj}+j-1)-\sum_{i=1}^{k-1}(x_{k-1,i}+i-1),
\end{equation}
where
\begin{equation}\label{eq: defintion of the rational functions a}
a_{ki}^\pm=\mp\frac{\prod_{j=1}^{k\pm1}(x_{k\pm1,j}-x_{ki})}{\prod_{j\neq i}(x_{kj}-x_{ki})},
\end{equation}
and $E_{k}^+=E_{k,k+1}$, $E_k^-=E_{k+1,k}$ where
$E_{ij}$ denotes the matrix units.
Let $U_n=\varphi(U(\gl_n))$. The algebra $U_n$ is a Galois $\Gamma$-order.

\subsection{Defining \texorpdfstring{$\Anlg{\gl_n}$}{A(gln)}}
Let $\mathbb{A}_n=A_1\times A_2\times\cdots\times A_n$ and
\begin{equation}\label{eq: def of Gamma Tilde}
\widetilde{\Gamma}=\Lambda^{\dl\mathbb{A}_n}
=\C[e_{ki},\mathcal{V}_\ell\mid1\leq i\leq k\leq n,2\leq\ell\leq n].  
\end{equation}
Here
\begin{equation}
\mathcal{V}_\ell=\mathcal{V}_\ell(x_{\ell1},\ldots,x_{\ell\ell})=\prod_{i<j}(x_{\ell i}-x_{\ell j})
\end{equation}
denotes the Vandermonde polynomial in the $\ell$ variables
$x_{\ell1},\ldots,x_{\ell\ell}$. Abstractly, $\widetilde{\Gamma}$ is isomorphic to the following quotient
\[
\C[T_{ki},V_\ell\mid1\leq i\leq k\leq n,2\leq\ell\leq n]/
(V_\ell^2-D_\ell(T_{\ell1},\ldots,T_{\ell\ell})\mid 2\leq\ell\leq n),  
\]
where $T_{ki}$,$V_\ell$ are indeterminates and $D_\ell(T_{\ell1},\ldots,T_{\ell\ell})$
is the Vandermonde discriminant. Also, let
$\widetilde{K}=\Frac(\widetilde{\Gamma})$ and
$\widetilde{\mathscr{K}}=(L\#\mathscr{M})^{\dl\mathbb{A}_n}$.

\begin{definition}\label{def: A(gln)}
Consider the following extension of $U(\gl_n)$, denoted $\Anlg{\gl_n}$, defined as
the subalgebra of $\widetilde{\mathscr{K}}$ generated
by $U_n\cup\{\mathcal{V}_2,\mathcal{V}_3,\cdots,\mathcal{V}_n\}$. Explicitly, $\Anlg{\gl_n}$
is the subalgebra of $\mathscr{L}$ generated by
\[
\begin{array}{ll}
X_{k}^\pm=\sum_{i=1}^k(\delta^{ki})^{\pm1}a_{ki}^\pm &\text{ for }k=1,\ldots,n-1,\\
X_{kk}=\sum_{j=1}^k(x_{kj}+j-1)-\sum_{i=1}^{k-1}(x_{k-1,i}+i-1) &\text{ for }k=1,\ldots,n,\\
\mathcal{V}_k=\mathcal{V}_k(x_{k1},\ldots,x_{kk})=\prod_{i<j}(x_{ki}-x_{kj}) &\text{ for }k=1,\ldots,n-1,
\end{array}
\]
where $a_{ki}^\pm$ are defined in (\ref{eq: defintion of the rational functions a}).
\end{definition}

The following proposition lists some basic properties of $\Anlg{\gl_n}$.
\begin{proposition}\label{prop: basic properties of A(gln)}
\item
\begin{enumerate}[\rm (i)]
\item $U(\gl_n)\cong U_n\subset\Anlg{\gl_n}$. \label{item: A(gln) contains U(gln)}
\item $\Anlg{\gl_n}$ is a Galois $\widetilde{\Gamma}$-ring. \label{item: A(gln) is a Galois ring}
\item $\mathcal{V}_n$ is central in $\Anlg{\gl_n}$. \label{item: Vn is central in A(gln)}
\item $Z(\Anlg{\gl_n})\cong\C[x_1,\ldots,x_n]^{A_n}$.\label{item: HC isomorphism for A(gln)}
\item There is a chain of subalgebras
$\Anlg{\gl_1}\subset\Anlg{\gl_2}\subset\cdots\subset\Anlg{\gl_n}$.\label{item: subalgebra chain for A(gln)}
\item $\Anlg{\gl_n}$ is the minimal extension of $U(\gl_n)$ with properties
\ref{item: HC isomorphism for A(gln)} and
\ref{item: subalgebra chain for A(gln)}.\label{item: A(gln) is minimal extension with given properties}
\end{enumerate}
\end{proposition}
\begin{proof}
\ref{item: A(gln) contains U(gln)} Clear because $\varphi$ is injective and $\Anlg{\gl_n}$
contains $\varphi(E_{k}^\pm),\varphi(E_kk)$.\\
\ref{item: A(gln) is a Galois ring} Define $\mathscr{X}$ as follows:
\[
\mathscr{X}=\{X_i^\pm,X_{ii},\mathcal{V}_j\mid1\leq i\leq n, 2\leq j\leq n\}.
\]
Since $X_i^\pm\in\mathscr{X}$, it is clear that $\bigcup_{u\in\mathscr{X}}\supp u$ generates
$\mathscr{M}$. Thus, $\Anlg{\gl_n}$ is a Galois $\widetilde{\Gamma}$-ring for every $n\geq1$ by
Proposition \ref{prop: Gamma Ring Alt Conditions}.\\
\ref{item: Vn is central in A(gln)} As $\delta^{ki}$ fixes $x_{\ell j}$ iff $\ell\neq k$
and $k\neq n$, it follows that $\mathcal{V}_n$ is central in $\Anlg{\gl_n}$.\\
\ref{item: HC isomorphism for A(gln)} We first show that
$Z(\Anlg{\gl_n}=\C\langle Z(U_n),\mathcal{V}_n\rangle$.
$\C\langle Z(U_n),\mathcal{V}_n)\rangle\subseteq Z(\Anlg{\gl_n})$ is clear. Next, we observe that
$\Anlg{\gl_n}\subset (L^\prime\#\mathscr{M})^{\mathbb{A}_n}$,
where $L^\prime=\C(x_{ki}\mid1\leq i\leq k\leq n-1)[x_{n1},\ldots,x_{nn}]$.
By Theorem \ref{thm: Center of a Galois Ring}, we have
\[
Z(\Anlg{\gl_n})=\Anlg{\gl_n}\cap\widetilde{K}^{\mathscr{M}}\subseteq
(L^\prime\#\mathscr{M})^{\mathbb{A}_n}\cap\widetilde{K}^{\mathscr{M}}\subseteq\C\langle Z(U_n),\mathcal{V}_n\rangle.
\]
Consider the
Harish-Chandra homomorphism
$\varphi_{\rm HC}\colon Z(U(\gl_n))\rightarrow\C[x_1,\ldots,x_n]^{S_n}$.
We can define an extension of this map
$\widetilde{\varphi}_{\rm HC}\colon Z(\Anlg{\gl_n})\rightarrow\C[x_1,\ldots,x_n]$
as follows:
\begin{equation}
\widetilde{\varphi}_{\rm HC}(X)=
\begin{cases}
\varphi_{\rm HC}(\varphi^{-1}(X)), & X\in Z(U_n),\\
\dl \prod_{1=i<j=n}(x_i-x_j), & X=\mathcal{V}_n.
\end{cases}
\end{equation}
In conjuction with Chevalley's Theorem (see \cite{Humpreys78}),
$\varphi_{HC}$ provides an isomorphism with $\C[x_1,\ldots,x_n]^{S_n}$. The
claim follows by recalling that $\C[x_1,\ldots,x_n]^{A_n}$ is generated
by the symmetric polynomials and the Vandermonde polynomial.\\
\ref{item: subalgebra chain for A(gln)} Follows from the definition.\\
\ref{item: A(gln) is minimal extension with given properties} We prove this
result by induction on $n$. Since $\Anlg{\gl_1}=U(\gl_1)$, the base step is clear.
Assuming the claim holds for $\Anlg{\gl_{n-1}}$, now consider an extension $\mathscr{A}$
of $U(\gl_n)$ satisfying \ref{item: HC isomorphism for A(gln)} and
\ref{item: subalgebra chain for A(gln)}. By \ref{item: subalgebra chain for A(gln)},
$\mathscr{A}$ contains $\mathcal{V}_\ell$ for $\ell=1,\ldots,n-1$, and it contains $U(\gl_n)$
by definition. From \ref{item: HC isomorphism for A(gln)} we get an element $\mathcal{V}$ that
is central in $\mathscr{A}$ that maps to $\prod_{i<j}(x_i-x_j)\in\C[x_1,\ldots,x_n]^{A_n}$.
This allows us to define an isomorphism $\tau\colon\mathscr{A}\rightarrow\Anlg{\gl_n}$
by sending $\{U(\gl_n),\mathcal{V}_\ell\mid\ell=1,\ldots,n-1\}$ to themselves
and $\mathcal{V}\mapsto\mathcal{V}_n$.
\end{proof}

\begin{remark}
In \cite{FSAlgofInvDifOp} another Galois algebra is described in the invariants of
a Weyl algebra with respect to a single alternating group in Corollary 24 in \cite{FSAlgofInvDifOp}.
\end{remark}

\section{The Structure of \texorpdfstring{$\Anlg{\gl_2}$}{A(gl2)}}\label{sec: n=2}

In this section, we find a presentation for $\Anlg{\gl_2}$ as an extension of
$U(\gl_2)$ and as a generalized Weyl algebra as well as prove that it is a Galois $\widetilde{\Gamma}$-order.

\begin{lemma}\label{lem: Case 2 two part lemma}
\text{}
\begin{enumerate}[\rm (i)]
\item $\mathcal{V}_2$ commutes with every element of $U_2$. \label{item: v2 commutes}
\item $\Anlg{\gl_2}=U_2\oplus(U_2\cdot\mathcal{V}_2)$ \label{item: direct sum decomp}
\end{enumerate}
\end{lemma}
\begin{proof}
\ref{item: v2 commutes} Follows by Proposition \ref{prop: basic properties of A(gln)} \ref{item: Vn is central in A(gln)}.\\
\ref{item: direct sum decomp} Since $\mathcal{V}_2$ commutes with everything in $U_2$,
\[
\Anlg{\gl_2}=\left\{\sum_{j=0}^\infty u_j\mathcal{V}_2^j\mid u_j\in U_2\text{, at most finitely many }u_j\neq0\right\}.
\]
Since $\mathcal{V}_2^2\in U_2$, $\Anlg{\gl_2}=U_2+U_2\cdot\mathcal{V}_2$. Now consider
$(12)_2:=((1),(12))\in\mathbb{S}_2$ acting on $\mathscr{L}$ by automorphisms. We have,
\[
(12)_2\vert_{U_2}=\Id\vert_{U_2},\quad(12)_2\vert_{U_2\cdot\mathcal{V}_2}=(-1)\cdot\Id\vert_{U_2\cdot\mathcal{V}_2}.
\]
This implies that $\Anlg{\gl_2}=U_2\oplus(U_2\cdot\mathcal{V}_2)$.
\end{proof}

\begin{definition}
The \emph{$k$-th Gelfand invariant for $\gl_n$} is defined as follows
\[
c_{nk}=\sum_{(i_1,i_2,\ldots,i_d)\in[n]^d}E_{i_1,i_2}E_{i_2,i_3}\cdots E_{i_{d-1},i_d}
E_{i_d,i_1}.
\]
There are $n$ such Gelfand invariants for $\gl_n$, and they generate the center
of $U(\gl_n)$.
\end{definition}

We now give a presentation for $\Anlg{\gl_2}$ in terms of $U(\gl_2)$.

\begin{proposition}\label{prop: n=2 isomorphism}
There is an isomorphism
\begin{equation}
\widetilde{\varphi}\colon
\frac{U(\gl_2)[T_2]}{\left(T_2^2-(-c_{21}^2+2c_{22}+1)\right)}\rightarrow\Anlg{\gl_2},
\end{equation}
where $T_2$ is an indeterminate and $c_{2i}$ are the Gelfand invariants for $\gl_2$. Explicitly
\begin{equation}\label{eq: n=2 iso images}
\widetilde{\varphi}\vert_{U(\gl_2)}=\varphi,\quad\widetilde{\varphi}(T_2)=\mathcal{V}_2,
\end{equation}
where $\varphi$ is the embedding from {\rm (\ref{eq: Ugln map})}.
\end{proposition}
\begin{proof}
Let $p(T_2)=T_2^2-(-c_{21}^2+2c_{22}-1)\in U(\gl_2)[T_2].$ Since $p(T_2)$ is degree two, $U(\gl_2)[T_2]/(p(T_2))$
is free of rank 2 as a left $U(\gl_2)$-module with basis $\{1,\vl{T_2}\}$ where $\vl{T_2}=T_2+(p(T_2))$. It follows Lemma
\ref{lem: Case 2 two part lemma} \ref{item: direct sum decomp} that $\Anlg{gl_2}$ is also free of rank 2 with basis
$\{1,\mathcal{V}_2\}$ via the isomorphism $\varphi$ in (\ref{eq: Ugln map}).
Therefore there is an isomorphism $\widetilde{\varphi}\colon U(\gl_2)[T_2]/(p(T_2))\rightarrow \Anlg{\gl_2}$
as $U(\gl_2)$-modules sending 1 to 1 and $\vl{T_2}$ to $\mathcal{V}_2$. Thus, it suffices to show that
$\widetilde{\varphi}(\vl{T_2}^2)=\mathcal{V}_2^2$.

\noindent To show this, we calculate the images of $c_{2i}$ under $\varphi$:
\begin{align*}
\varphi(c_{21})&=\varphi\big(E_{11}+E_{22}\big)=(x_{11})+(x_{21}+x_{22}-x_{11}+1)=x_{21}+x_{22}+1,\\
\varphi(c_{22})&=\varphi\big(E_{11}^2+E_1^+E_1^-+E_1^-E_1^++E_{22}^2\big)=x_{21}^2+x_{22}^2+x_{21}+x_{22}.
\end{align*}
As such,
\begin{align*}
\widetilde{\varphi}(\vl{T_2}^2)=\widetilde{\varphi}(-c_{21}^2+2c_{22}+1)&=-\varphi(c_{21})^2+2\varphi(c_{22})+1\\
&=-(x_{21}+x_{22}+1)^2+2(x_{21}^2+x_{22}^2+x_{21}+x_{22})+1\\
&=(x_{21}-x_{22})^2=\mathcal{V}_2^2.
\end{align*}
Therefore, $\widetilde{\varphi}$ is an
algebra isomorphism.
\end{proof}

\begin{theorem}\label{thm: Agl2 is a Galois order}
$\Anlg{\gl_2}$ is a Galois $\widetilde{\Gamma}$-order.
\end{theorem}
\begin{proof}
We first observe that $\Anlg{\gl_2}$ is a Galois $\widetilde{\Gamma}$-ring by
Proposition \ref{prop: basic properties of A(gln)} \ref{item: A(gln) is a Galois ring}.
To prove that $\Anlg{\gl_2}$ is a Galois $\widetilde{\Gamma}$-order,
we will use Theorem \ref{thm: Galois Order condition}. Since $\Gamma$ is a
Harish-Chandra subalgebra of $U(\gl_2)$, $\widetilde{\Gamma}$
is a Harish-Chandra subalgebra of $\Anlg{\gl_2}$. Since $\mathbb{A}_2$ is a group,
all we need to show is that $\widetilde{\Gamma}$ is maximal commutative
in $\Anlg{\gl_2}$. This is clear because $\Gamma$ is maximal commutative in $U_2$,
and $\widetilde{\Gamma}$ is just an extension by a central
element by Proposition \ref{prop: n=2 isomorphism}. $\widetilde{\Gamma}$
is maximal commutative in $\Anlg{\gl_2}$; therefore, $\Anlg{\gl_2}$ is a
Galois $\widetilde{\Gamma}$-order. 
\end{proof}

The following shows that $\Anlg{\gl_2}$ is a generalized Weyl algebra \cite{BavulaGWA},
which gives another way to show it is a Galois order \cite{FO10}. First we recall
the definition of a generalized Weyl algebra.

\begin{definition}[\cite{BavulaGWA}]
Let $D$ be a ring, $\sigma$ a ring automorphism of $D$, and $t$ a central element of
$D$. The \emph{generalized Weyl algebra} of rank 1, $D(\sigma,t)$
is a ring generated by the ring $D$ and two elements $x$ and $y$ subject to the
following relations:
\begin{equation}\label{eq: GWA defining relations 1}
xd=\sigma(d)x\text{ and }yd=\sigma^{-1}(d)y\text{ for all }d\in D;
\end{equation}
\begin{equation}\label{eq: GWA defining relations 2}
yx=t \text{ and } xy=\sigma(t).
\end{equation}
\end{definition}

\begin{proposition}\label{prop: A(gl2) is a GWA}
$\Anlg{\gl_2}$ is isomorphic to the generalized Weyl algebra 
$\widetilde{\Gamma}(\sigma,t)$, where $\sigma=\delta^{11}$,\\
$t=-e_{22}+e_{11}e_{21}-e_{11}^2$, and $\widetilde{\Gamma}$ is defined in
(\ref{eq: def of Gamma Tilde}).
\end{proposition}
\begin{proof}
Recall that $\Anlg{\gl_2}$ is the subalgebra generated by $\widetilde{\Gamma},X_1^\pm$ (see Definition \ref{def: A(gln)}).
We define $\psi\colon\widetilde{\Gamma}(\sigma,t)\rightarrow\Anlg{\gl_2}$ by
\[
x\mapsto X_1^+,~ y\mapsto X_1^-\text{, and }\gamma\mapsto\gamma\text{ for all }
\gamma\in\widetilde{\Gamma}.
\]
One can verify that the defining relations (\ref{eq: GWA defining relations 1})
and (\ref{eq: GWA defining relations 2}) are preserved by $\psi$, making it well-defined.
Clearly $\psi$ is surjective. For injectivity, we note $\psi$ is graded, when
$\widetilde{\Gamma}(\sigma,t)$ and $\Anlg{\gl_2}$ are equipped with the $\Z$-gradations determined
by
\begin{equation}
\deg X_1^\pm=\pm1,~\deg\gamma=0~\forall\gamma\in\widetilde{\Gamma},~\deg x=1,~\deg y=-1.
\end{equation}
As such, $\ker\psi$ is a graded ideal. But $\widetilde{\Gamma}\cap(\ker\psi)=0$. Since the only
graded ideal $I$ in a generalized Weyl algebra $D(\sigma,t)$, where $D$ is a domain, with
$D\cap I=0$ is $I=0$, we get $\ker\psi=0$.
\end{proof}

We observe the following interesting property of $\Anlg{\gl_2}$ that we prove
does not hold for general $n$ (see Proposition \ref{prop: Sn invariants larger than Ugln}).

\begin{proposition}\label{prop: S2 invariants of Agl2 is U2}
$\Anlg{\gl_2}$ has the property that $(\Anlg{\gl_2})^{\mathbb{S}_2}=U_2$.
\end{proposition}

\begin{proof}
This becomes clear when we consider the direct sum decomposition shown in
Lemma \ref{lem: Case 2 two part lemma}~\ref{item: direct sum decomp}.
Consider $a+b\mathcal{V}_2\in\Anlg{\gl_2}$:
\begin{align*}
a+b\mathcal{V}_2\in(\Anlg{\gl_2})^{\mathbb{S}_2}&\iff(12)_2(a+b\mathcal{V}_2)=a+b\mathcal{V}_2\\
&\iff a-b\mathcal{V}_2=a+b\mathcal{V}_2\\
&\iff b=0\\
&\iff a+b\mathcal{V}_2=a\in U_2.\\
\end{align*}
Therefore, $(\Anlg{\gl_2})^{\mathbb{S}_2}=U_2$.
\end{proof}

\section{The Structure of \texorpdfstring{$\Anlg{\gl_3}$}{A(gl3)}}\label{sec: n=3}
Based on the result of the previous section, the next logical step is to see if similar
results hold for $\gl_n$ with $n\geq3$. We will continue using the notation of
the images of the generators of the $U(\gl_n)$ as before. As such:
\[
X_i^\pm:=\varphi(E_i^\pm)\quad\text{and}\quad X_{ii}:=\varphi(E_{ii}).
\]
\subsection{Non-polynomial rational functions in \texorpdfstring{$\Anlg{\gl_3}$}{A(gl3)}} 
Unlike in $U(\gl_3)$ and $\Anlg{\gl_2}$, we can construct non-polynomial rational
functions in $\Anlg{\gl_3}$. It follows that for $n\geq3$, $\Anlg{\gl_n}$ is not a Galois
$\widetilde{\Gamma}$-order, and the invariant property of $\Anlg{\gl_2}$ does not hold.

\begin{lemma}
The following identity holds in $\Anlg{\gl_3}$:
\begin{equation}\label{eq: V2 commutator identity}
\pm[X_{2}^\pm,\mathcal{V}_2]=(\delta^{21})^{\pm1}a_{21}^{\pm}-(\delta^{22})^{\pm1}a_{22}^{\pm}.
\end{equation}
\end{lemma}
\begin{proof}
To show this, consider $\mathcal{V}_2X_2^\pm$:
\begin{align*}
\mathcal{V}_2X_2^\pm&=(x_{21}-x_{22})((\delta^{21})^{\pm1} a_{21}^\pm+(\delta^{22})^{\pm1} a_{22}^\pm)\\
&=(\delta^{21})^{\pm1} a_{21}^\pm(x_{21}\pm1-x_{22})+(\delta^{22})^{\pm1} a_{22}^\pm(x_{21}-x_{22}\mp1)\\
&=X_2^\pm\mathcal{V}_2\pm((\delta^{21})^{\pm1} a_{21}^\pm-(\delta^{22})^{\pm1} a_{22}^\pm).
\end{align*}
Therefore, $\pm[X_2^\pm,\mathcal{V}_2]=(\delta^{21})^{\pm1} a_{21}^\pm-(\delta^{22})^{\pm1} a_{22}^\pm$.
\end{proof}

Let us denote the element described in (\ref{eq: V2 commutator identity})
by $\widetilde{X}_2^\pm$. We define the following:
\begin{center}
\begin{tabular}{l l}
$A_{21}^+:=\frac{1}{2}(X_2^+ +\widetilde{X}_2^+)=\delta^{21}a_{21}^+$
& $A_{21}^{-}:=\frac{1}{2}(X_2^-+\widetilde{X}_2^-)=(\delta^{21})^{-1}a_{21}^-$\\
$A_{22}^{+}:=\frac{1}{2}(X_2^+-\widetilde{X}_2^+)=\delta^{22}a_{22}^+$
& $A_{22}^{-}:=\frac{1}{2}(X_2^--\widetilde{X}_2^-)=(\delta^{22})^{-1}a_{22}^-$
\end{tabular}
\end{center}
By their definition, it is clear that they are in $\Anlg{\gl_3}$.

The following example shows that if $n\geq3$, then $\widetilde{\Gamma}$
is not maximal commutative; hence, $\Anlg{\gl_n}$ is not a Galois $\widetilde{\Gamma}$-order
by Proposition \ref{prop: Gamma maxl comm in a Galois Order}.

\begin{example}\label{ex: Gammatilde is not maximal comm}
The following element belongs to $\Anlg{\gl_n}$ for $n\geq3$:
\[
A_{21}^+A_{21}^-=-\frac{\prod_{i=1}^3(x_{3i}-x_{21}+1)}{(x_{22}-x_{21}+1)}\cdot\frac{x_{11}-x_{21}}{x_{22}-x_{21}}.
\]
This is a rational function; hence, it lies in ${\rm Cent}_{\Anlg{\gl_3}}(\widetilde{\Gamma})$.
\end{example}

The following rather surprising fact shows that
the property in Proposition \ref{prop: S2 invariants of Agl2 is U2} does not hold for
larger $n$.

\begin{proposition}\label{prop: Sn invariants larger than Ugln}
For $n\geq3$, $\Anlg{\gl_n}^{\mathbb{S}_n}\supsetneq U_n$.
\end{proposition}
\begin{proof}
The fact that $U_n\subset\Anlg{\gl_n}^{\mathbb{S}_n}$ is obvious by definition. To
show the containment is strict, we recall that because $U_n$ is a Galois $\Gamma$-order,
it is known that $U_n\cap K=\Gamma$. Therefore, we consider
$\Anlg{\gl_n}^{\mathbb{S}_n}\cap K$. Since $U_3\subseteq U_n$ for every
$n\geq3$, it suffices to show that $\Anlg{\gl_3}^{\mathbb{S}_3}\cap K\supsetneq\Gamma$.

The object to prove this claim is constructed in the same way as in Example
\ref{ex: Gammatilde is not maximal comm}. It is quickly observed that 
\[
A_{21}^+A_{21}^-A_{22}^+A_{22}^-=\frac{\prod_{i=1}^3(x_{3i}-x_{21}+1)}{(x_{22}-x_{21}+1)}\cdot\frac{x_{11}-x_{21}}{x_{22}-x_{21}}\cdot
\frac{\prod_{i=1}^3(x_{3i}-x_{22}+1)}{(x_{21}-x_{22}+1)}\cdot\frac{x_{11}-x_{22}}{x_{21}-x_{22}}
\]
is invariant under the action of $\mathbb{S}_3$. This element is clearly
not in $\Gamma$, so this element is in $\Anlg{\gl_3}^{\mathbb{S}_3}\cap K\setminus\Gamma$,
thereby proving the claim.
\end{proof}

\subsection{Generators and relations for \texorpdfstring{$n=3$}{n=3}}\label{subsec: partial presentation}

Based on the previous subsection, we determine a set of generators and some verified relations
for $\Anlg{\gl_3}$. However, we do not know if this constitutes a presentation, that is, this
may be an incomplete list.

\begin{proposition}\label{prop: A(gl3) relations}
The algebra $\Anlg{\gl_3}$ is generated by $\{X_{11},X_{22},X_{33},A_{11}^\pm,A_{21}^\pm,A_{22}^\pm,\mathcal{V}_2,\mathcal{V}_3\}$, where $A_{ij}^\pm:=(\delta^{ij})^{\pm1}a_{ij}^\pm$, $\mathcal{V}_2=x_{21}-x_{22}$, and
$\mathcal{V}_3=\prod_{i<j}(x_{3i}-x_{3j})$.
What follows is a list of known relations:
\begin{enumerate}[\rm (i)]
\item $[\mathcal{V}_3,X]=0$ for all $X\in\Anlg{\gl_3}$ (i.e $\mathcal{V}_3$ is central
in $\Anlg{\gl_3}$),\label{rel: V3 is central in A(gl3)}
\item $[X,Y]=0$ for all $X,Y\in\mathfrak{h}=\Span_{\C}\{X_{11},X_{22},X_{33},\mathcal{V}_2,\mathcal{V}_3\}$,\label{rel: cartan like algebra for A(gl3)}
\item $[h,A_{ij}^\pm]=\pm\alpha_{ij}(h)A_{ij}^\pm$ for all $h\in\mathfrak{h}$ and $1\leq j\leq i\leq2$,
where $\alpha_{ij}(h)$ are given by the following matrix:
\[
\begin{blockarray}{cccccc}
 & X_{11} & X_{22} & X_{33} & \mathcal{V}_2 & \mathcal{V}_3\\
\begin{block}{c[ccccc]}
\alpha_{11} & 1 & -1 & 0 & 0 & 0\\
\alpha_{21} & 0 & 1 & -1 & 1 & 0\\
\alpha_{22} & 0 & 1 & -1 & -1 & 0\\
\end{block}
\end{blockarray},
\]\label{rel: chevalley relations in A(gl3)}
\item $[A_{21}^\pm,A_{22}^\mp]=0$,\label{rel: 21 and 22 with opposite powers commute}
\item $[A_{11}^\pm,A_{2i}^\mp]=0$ for $i=1,2$,\label{rel: 11 and 2i with opposite powers commute}
\item $[A_{11}^+,A_{11}^-]=X_{11}-X_{22}$, \label{rel: fourth U(gl3) relation}
\item $[A_{21}^+,A_{21}^-]+[A_{22}^+,A_{22}^-]=X_{22}-X_{33}$, \label{rel: 21+22 is polynomial}
\item $[A_{11}^\pm,[A_{11}^\pm,A_{2i}^\pm]]=0$ for $i=1,2$,\label{rel: serre type relation in A(gl3)}
\item $A_{22}^\pm\mathcal{V}_2 A_{21}^\pm=A_{21}^\pm\mathcal{V}_2A_{22}^\pm$.\label{rel: weird relation in A(gl3)}
\end{enumerate}
\end{proposition}
\begin{proof}
Any of the relations involving only elements from $U(\gl_3)$ (such as
\ref{rel: fourth U(gl3) relation}) follow from $U(\gl_3)$ relations by
recalling that $\{X_{11},X_{22},X_{33},A_{11}^+,A_{11}^-\}\in\Anlg{\gl_3}$
correspond to $\{E_{11},E_{22},E_{33},E_{12},E_{21}\}\in U(\gl_3)$. All that
remains is to prove the relations involving new elements.\\
\ref{rel: V3 is central in A(gl3)} This follows from Proposition \ref{prop: basic properties of A(gln)}
\ref{item: Vn is central in A(gln)}.\\
\ref{rel: cartan like algebra for A(gl3)} This follows by observing that each is an element of
$\widetilde{\Gamma}$ which is a commutative ring.\\
\ref{rel: chevalley relations in A(gl3)} By the statement at the beginning
of this proof and \ref{rel: V3 is central in A(gl3)}, we only need to check
the second two rows and the second to last column. Each is proved in an
identical manner, we provide one below:
\begin{align*}
\mathcal{V}_2\cdot A_{21}^+&=(x_{21}-x_{22})\cdot-\delta^{21}\frac{\prod_{i=1}^{3}x_{3i}-x_{21}}{x_{22}-x_{21}}\\
&=-\delta^{21}\frac{\prod_{i=1}^{3}x_{3i}-x_{21}}{x_{22}-x_{21}}\cdot(x_{21}-x_{22}+1)\\
&=A_{21}^+\mathcal{V}_2+A_{21}^+.
\end{align*}
Thus, $[\mathcal{V}_2,A_{21}^+]=A_{21}^+=\alpha_{21}(\mathcal{V}_2)A_{21}^+$.\\
\ref{rel: 21 and 22 with opposite powers commute}
Consider the following calculation:
\begin{align*}
A_{21}^+A_{22}^-&=-\delta^{21}\frac{\prod_{i=1}^{3}x_{3i}-x_{21}}{x_{22}-x_{21}}
\cdot(\delta^{22})^{-1}\frac{x_{11}-x_{22}}{x_{21}-x_{22}}\\
&=-\delta^{21}(\delta^{22})^{-1}\frac{\prod_{i=1}^{3}x_{3i}-x_{21}}{x_{22}-x_{21}-1}\cdot
\frac{x_{11}-x_{22}}{x_{21}-x_{22}}\\
&= -\delta^{21}(\delta^{22})^{-1}\frac{\prod_{i=1}^{3}x_{3i}-x_{21}}{x_{22}-x_{21}}\cdot
\frac{x_{11}-x_{22}}{x_{21}-x_{22}+1}\\
&=(\delta^{22})^{-1}\frac{x_{11}-x_{22}}{x_{21}-x_{22}}
\cdot-\delta^{21}\frac{\prod_{i=1}^{3}x_{3i}-x_{21}}{x_{22}-x_{21}}\\
&=A_{22}^-A_{21}^+.
\end{align*}
The other relation is proved similarly.\\
\ref{rel: 11 and 2i with opposite powers commute} Consider the following calculation:
\begin{align*}
A_{11}^+A_{22}^-&=-\delta^{11}(x_{21}-x_{11})(x_{22}-x_{11})
\cdot(\delta^{22})^{-1}\frac{x_{11}-x_{22}}{x_{21}-x_{22}}\\
&=-\delta^{11}(\delta^{22})^{-1}(x_{21}-x_{11})(x_{22}-x_{11}-1)
\cdot\frac{x_{11}-x_{22}}{x_{21}-x_{22}}\\
&= -\delta^{11}(\delta^{22})^{-1}(x_{21}-x_{11})(x_{22}-x_{11})
\cdot\frac{x_{11}-x_{22}+1}{x_{21}-x_{22}}\\
&=(\delta^{22})^{-1}\frac{x_{11}-x_{22}}{x_{21}-x_{22}}
\cdot-\delta^{11}(x_{21}-x_{11})(x_{22}-x_{11})\\
&=A_{22}^-A_{11}^+.
\end{align*}
The other relations are proved similarly.\\
\ref{rel: 21+22 is polynomial} We consider the relation
$[E_{23},E_{32}]=E_{22}-E_{33}$ mapped under $\varphi$ from (\ref{eq: Ugln map}):
\begin{align*}
X_{22}-X_{33}&=[X_2^+,X_2^-]\\
&=[A_{21}^+ + A_{22}^+,A_{21}^-+A_{22}^-]\\
&=[A_{21}^+,A_{21}^-]+[A_{21}^+,A_{22}^-]+[A_{22}^+,A_{21}^-]+[A_{22}^+,A_{22}^-]\\
&=[A_{21}^+,A_{21}^-]+[A_{22}^+,A_{22}^-]\quad\text{by \ref{rel: 21 and 22 with opposite powers commute}}.
\end{align*}
This demonstrates that \ref{rel: 21+22 is polynomial} holds.\\
\ref{rel: serre type relation in A(gl3)} We observe that
\begin{align*}
A_{11}^-A_{22}^-&=(\delta^{11})^{-1}
\cdot(\delta^{22})^{-1}\frac{x_{11}-x_{22}}{x_{21}-x_{22}}\\
&=(\delta^{11}\delta^{22})^{-1}\frac{x_{11}-x_{22}}{x_{21}-x_{22}}\\
&=(\delta^{22})^{-1}\frac{x_{11}-x_{22}+1}{x_{21}-x_{22}}\cdot(\delta^{11})^{-1}\\
&=A_{22}^- A_{11}^- -(\delta^{11}\delta^{22})^{-1}\frac{1}{x_{21}-x_{22}}\\
[A_{11}^-,A_{22}^-]&=-(\delta^{11}\delta^{22})^{-1}\frac{1}{x_{21}-x_{22}},
\end{align*}
which has no $x_{11}$'s and as such commutes with $A_{11}^-$. Thus, $[A_{11}^-,[A_{11}^-,A_{22}^-]]=0$.
The others are proved identically.\\
\ref{rel: weird relation in A(gl3)} We prove this by direct computation as follows:
\begin{align*}
A_{22}^\pm\mathcal{V}_2A_{21}^\pm
&=(\delta^{21}\delta^{22})^{\pm1}\frac{\prod_{i=1}^{2\pm1}(x_{2\pm1,i}-x_{21})(x_{2\pm1,i}-x_{22})}{x_{21}-x_{22}}\\
&=-(\delta^{21})^{\pm1}\prod_{i=1}^{2\pm1}x_{2\pm1,i}-x_{21}\cdot (\delta^{22})^{\pm1}\frac{\prod_{i=1}^{2\pm1}x_{2\pm1,i}-x_{22}}{x_{21}-x_{22}}\\
&=(\delta^{21})^{\pm1}\prod_{i=1}^{2\pm1}x_{2\pm1,i}-x_{21}\cdot\frac{\mathcal{V}_2}{-\mathcal{V}_2}(\delta^{22})^{\pm1}\frac{\prod_{i=1}^{2\pm1}x_{2\pm1,i}-x_{22}}{x_{21}-x_{22}}\\
&=A_{21}^\pm\mathcal{V}_2A_{22}^\pm.
\end{align*}
This verifies that relation \ref{rel: weird relation in A(gl3)} holds.
\end{proof}

\begin{openprob}
Determine whether the relations in Proposition \ref{prop: A(gl3) relations}
constitute a presentation for the algebra $\Anlg{\gl_3}.$
\end{openprob}


\section{Finite-Dimensional Modules over \texorpdfstring{$\Anlg{\gl_n}$}{A(gln)}}\label{sec: fd modules over A(gln)}

Since, as was shown in Section \ref{sec: n=3}, $\Anlg{\gl_n}$ is not a Galois
$\widetilde{\Gamma}$-order, techniques different from \cite{FO14} are required
to study representations of $\Anlg{\gl_n}$.

If we consider the case of $n=2$, we recall that 
$\Anlg{\gl_2}\cong U(\gl_2)[T_2]/(T_2^2-(-c_{21}^2+2c_{22}+1))$. As such, it makes
sense to consider the induction and restriction functors between the categories of
$\Anlg{\gl_2}$-modules and $U(\gl_2)$-modules.

By applying the restriction functor to a given finite-dimensional simple module,
we see that it decomposes to a direct sum of finite-dimensional simple
$U(\gl_2)$-modules, so the induction functor should help us to construct
all of the possible finite-dimensional simple $\Anlg{\gl_2}$-modules.

\begin{proposition}\label{prop: f.d. simple A(gl2) mod}
The finite-dimensional simple $\Anlg{\gl_2}$-modules are characterized by ordered pairs
$(\lambda_2,\varepsilon_2)$, where $\lambda_2:=(\lambda_{21},\lambda_{22})\in\C^2$
is a dominant integral weight for $U(\gl_2)$ (i.e. $\lambda_{21}-\lambda_{22}\in\Z_{\geq0}$) and $\varepsilon_2\in\{1,-1\}$.
\end{proposition}
\begin{proof}
Recall that every finite-dimensional simple $U(\gl_2)$-module is characterized by a weight denoted by a pair
of complex numbers $\lambda_2=(\lambda_{21},\lambda_{22})$ with $\lambda_{21}-\lambda_{22}\in\Z_{\geq0}$; we will denote this module by
$V(\lambda_2)$. We can induce such a module $V(\lambda_2)$
to a $\Anlg{\gl_2}$-module as follows,
\[
\Anlg{\gl_2}\otimes_{U(\gl_2)} V(\lambda_2).
\]
So, it is important to describe $\Anlg{\gl_2}$ as a right $U(\gl_2)$-module. By Proposition \ref{prop: n=2 isomorphism}:
\[
\Anlg{\gl_2}\cong\frac{U(\gl_2)[T_2]}{(T_2^2-(-c_{21}^2+2c_{22}+1))}
\cong U(\gl_2)\oplus T_2U(\gl_2)
\]
as right $U(\gl_2)$-modules. Thus:
\begin{align*}
\Anlg{\gl_2}\otimes_{U(\gl_2)} V(\lambda_2)
&\cong \big(U(\gl_2)\oplus T_2U(\gl_2)\big)\otimes_{U(\gl_2)} V(\lambda_2)\\
&\cong \big(U(\gl_2)\otimes_{U(\gl_2)} V(\lambda_2)\big)
\oplus\big(T_2U(\gl_2)\otimes_{U(\gl_2)} V(\lambda_2)\big)\\
&\cong\big(1\otimes_{U(\gl_2)}V(\lambda_2)\big)
\oplus\big(T_2\otimes_{U(\gl_2)}V(\lambda_2)\big).
\end{align*}
As such, we can determine the action of $T_2$ on this modules now. For
$v\in V(\lambda_2)$, we have that $T_2.(1\otimes v)=T_2\otimes v$, and
$T_2.(T_2\otimes v)=T_2^2\otimes v=1\otimes T_2^2.v=(\lambda_{21}-\lambda_{22})^2(1\otimes v)$.
Thus, $T_2$ can be characterized by the following matrix:
\[
\begin{bmatrix}
0 & (\lambda_{21}-\lambda_{22})^2I\\
I & 0
\end{bmatrix}
\cong
\begin{bmatrix}
(\lambda_{21}-\lambda_{22})I & 0\\
0 & -(\lambda_{21}-\lambda_{22})I
\end{bmatrix},
\]
so we can see that $\Anlg{\gl_2}\otimes_{U(\gl_2)}V(\lambda_2)$ decomposes into
the two eigenspaces of the action of $T_2$: $V(\lambda_2,+1)
:=\langle(\lambda_{21}-\lambda_{22})(1\otimes v)+T_2\otimes v\mid
v\in V(\lambda_2)\rangle$ and $V(\lambda_2,-1)
:=\langle-(\lambda_{21}-\lambda_{22})(1\otimes v)+T_2\otimes v\mid
v\in V(\lambda_2)\rangle$ both of which are clearly simple. It is also
clear that as vector spaces $V(\lambda_2,\pm1)\cong V(\lambda_2)$.

Conversely, if we have a finite-dimensional simple $\Anlg{\gl_2}$-module $V$ restricted
to a $U(\gl_2)$-module, it must remain simple, as $T_2$ is a central element. As such,
$V\cong V(\lambda_2)$ for some weight $\lambda_2$. Thus,
$V\cong V(\lambda_2,\varepsilon_2)$ for some $\varepsilon_2\in\{\pm1\}$.
\end{proof}

Next, we classify a collection of finite-dimensional simple weight modules over $\Anlg{\gl_n}$.

\begin{definition}\label{def: A(gln) class of fd simple modules}
Let $V(\lambda_n)$ be a weight module of $U(\gl_n)$, we extend it to a module for $\Anlg{\gl_n}$, denoted
$V(\lambda_n,\varepsilon_n,\varepsilon_{n-1},\ldots,\varepsilon_2)$, by describing the actions of each $\mathcal{V}_k$ for
$k=2,3,\ldots,n$ as follows:
\[
\mathcal{V}_n.v=\varepsilon_{n}\prod_{i\leq j}(\lambda_{ni}-\lambda_{nj}+j-i)v,
\]
with $\varepsilon_n=\pm1$. Recall that when we restrict $V(\lambda_n)$ to a $U(\gl_k)$ module, the number of
simple $U(\gl_k)$ modules it decomposes into is the same as the number of ways to fill in the $k$-th row of a Gelfand-Tsetlin
pattern with top row $\lambda_n$. Denote this number by $r_{\lambda_n,k}$. Then let $\mathcal{V}_k$ act diagonallizably
on a $v=(v_1,\ldots,v_{r_{\lambda_n,k}})\in V(\lambda_n,\varepsilon_n,\varepsilon_{n-1},\ldots,\varepsilon_2)$ by the following $r_{\lambda_n,k}\times r_{\lambda_n,k}$ matrix,
\[
\begin{pmatrix}
\varepsilon_{k,1} \prod_{i\leq j}(\lambda_{ki}^1-\lambda_{kj}^1+j-i) & 0 & \cdots & 0\\
0 & \varepsilon_{k,2} \prod_{i\leq j}(\lambda_{ki}^2-\lambda_{kj}^2+j-i) & \cdots & 0\\
 & & \ddots & \\
0 & \cdots & &\varepsilon_{k,r_{\lambda_n,k}} \prod_{i\leq j}(\lambda_{ki}^{r_{\lambda_n,k}}-\lambda_{kj}^{r_{\lambda_n,k}}+j-i)
\end{pmatrix},
\]
where $\lambda_{ki}^\ell$ denotes the $ki$ entry from the $\ell$-th pattern in the decomposition of $v$ as a $U(\gl_k)$-module,
and $\varepsilon_{k}=(\varepsilon_{k,1},\varepsilon_{k,2},\ldots,\varepsilon_{k,r_{\lambda_n,k}})\in\{\pm1\}^{r_{\lambda_n,k}}$.
\end{definition}

\begin{theorem}\label{thm: A(gln) finite dim simple modules}
Every finite-dimensional simple module over $\Anlg{\gl_n}$, on which
$\mathcal{V}_2,\ldots,\mathcal{V}_{n-1}$ act diagonallizably, is of the form
$V(\lambda_n,\varepsilon_n,\varepsilon_{n-1},\ldots,\varepsilon_2)$ (see Definition \ref{def: A(gln) class of fd simple modules}), where
$\lambda_n=(\lambda_{n1},\lambda_{n2},\ldots,\lambda_{nn})$ is a dominant integral weight of $U(\gl_n)$,
$\varepsilon_j\in\{\pm1\}^{r_{\lambda_n,j}}$, with $r_{\lambda_n,j}$ denoting the number
of ways to fill the $j$-th row of Gelfand-Tsetlin pattern with fixed top row $\lambda_n$,
and $j=2,3,\ldots,n$.
\end{theorem}
\begin{proof}
We prove this by induction on $n$. For the base case, $n=3$, we have the
following commutative diagram:
\begin{figure}[!ht]
\centering
\begin{tikzpicture}
\node(A3)                       {$\Anlg{\gl_3}\fdMod$};
\node(A2)   [right=1cm of A3]   {$\Anlg{\gl_2}\fdMod$};
\node(U3)   [below=1cm of A3]   {$U(\gl_3)\fdMod$};
\node(U2)   [below=1cm of A2]   {$U(\gl_2)\fdMod$};
\foreach \x in {2,3} {
\draw[->,>=stealth,black] (A\x) -- (U\x);
}
\foreach \x in {U,A} {
\draw[->,>=stealth,black] (\x3) -- (\x2);
}
\end{tikzpicture},
\end{figure}

\noindent where each arrow is the restriction functor. If we consider a simple 
$V\in\Anlg{\gl_3}\fdMod$ and its image in the bottom right corner, we see that
$V\cong\bigoplus_{\lambda_3}\bigoplus_{\lambda_2}V(\lambda_2)_{\lambda_3}
\in U(\gl_2)\fdMod$, where $\lambda_3$ and $\lambda_2$ are weights for $U(\gl_3)$ and
$U(\gl_2)$, respectively, by the semi-simplicity of $U(\gl_3)$ and $U(\gl_2)$. Moreover,
$V(\lambda_2)_{\lambda_3}$'s are the components of the restriction of $V(\lambda_3)$
to $U(\gl_2)$. We know that $\mathcal{V}_2$ must have a diagonal action by assumption.
As such, we have $V\cong\bigoplus_{\lambda_3}\bigoplus_{\lambda_2}
V(\lambda_2,\varepsilon_2)_{\lambda_3}$ in the upper right corner by Proposition
\ref{prop: f.d. simple A(gl2) mod}, where $\varepsilon_2=\varepsilon_2(\lambda_2)$
depends $\lambda_2$. This is because otherwise the dimensions of the $\lambda_2$ weight
spaces would not match. Since $\mathcal{V}_2$ acts diagonally, $\mathcal{V}_3$ is central, and
the diagram commutes, it follows that
$V\cong V(\lambda_3,\varepsilon_3,\varepsilon_2)\in\Anlg{\gl_3}\fdMod$, where
$\varepsilon_3$ is determined as in Proposition \ref{prop: f.d. simple A(gl2) mod},
and $\varepsilon_2=\{\varepsilon_2(\lambda_2)\}_{\lambda_2}$ is indexed by the
number $r_{\lambda_3,2}$.

To finish the induction we look at a similar diagram:
\begin{figure}[!ht]
\centering
\begin{tikzpicture}
\node(A3)                       {$\Anlg{\gl_n}\fdMod$};
\node(A2)   [right=1cm of A3]   {$\Anlg{\gl_{n-1}}\fdMod$};
\node(A1)   [right=0.5cm of A2]   {$\cdots$};
\node(A0)   [right=0.5cm of A1]   {$\Anlg{\gl_2}\fdMod$};
\node(U3)   [below=1cm of A3]   {$U(\gl_n)\fdMod$};
\node(U2)   [below=1cm of A2]   {$U(\gl_{n-1})\fdMod$};
\node(U1)   [below=1.25cm of A1]   {$\cdots$};
\node(U0)   [below=1cm of A0]   {$U(\gl_2)\fdMod$};
\foreach \x in {0,2,3} {
\draw[->,>=stealth,black] (A\x) -- (U\x);
}
\foreach \x in {U,A} {
\draw[->,>=stealth,black] (\x3) -- (\x2);
\draw[->,>=stealth,black] (\x2) -- (\x1);
\draw[->,>=stealth,black] (\x1) -- (\x0);
}
\end{tikzpicture}
\end{figure}

Following the image of a simple $V\in\Anlg{\gl_n}\fdMod$ and using identical arguments,
we observe that:
\[
V\cong\bigoplus_{\lambda_n}\bigoplus_{\lambda_{n-1}}V(\lambda_{n-1})_{\lambda_n}
\in U(\gl_{n-1})\fdMod.
\]
By the induction hypothesis,
\[
V\cong\bigoplus_{\lambda_n}\bigoplus_{\lambda_{n-1}}
V(\lambda_{n-1},\varepsilon_{n-1},\varepsilon_{n-2},
\ldots,\varepsilon_2)_{\lambda_n}\in \Anlg{\gl_{n-1}}\fdMod.
\]
Finally by $\mathcal{V}_n$ central, $\mathcal{V}_j$ acting diagonally for $j=2,\ldots,n-1$,
and the diagram commuting, it follows that
$V\cong V(\lambda_n,\varepsilon_n,\varepsilon_{n-1},\ldots,\varepsilon_2)$.
\end{proof}

The following example demonstrates that $\Anlg{\gl_n}\fdMod$ is not semi-simple for every
$n\geq2$.

\begin{example}\label{ex: counterexample to ss of fd Agln-mod}
We recall that $\mathcal{V}_2^2$ must act diagonally
on any $\Anlg{\gl_2}$-module $V$ because $\Res{U(\gl_2)}{\Anlg{\gl_2}}V$
can be viewed as a direct sum of irreducible $U(\gl_2)$-modules and
$\mathcal{V}_2^2$ is a quadratic polynomial of Gelfand invariants in
$U(\gl_2)$. Let $V=V(0)\oplus V(0)$, where $U(\gl_2)$ acts trivially. This means that
$\mathcal{V}_2^2$ must act as $\Id_V$. We define the following action of $\mathcal{V}_2$
\[
\mathcal{V}_2.
\begin{pmatrix}
v_1\\
v_2
\end{pmatrix}
=\begin{pmatrix}
1 & \alpha\\
0 & -1
\end{pmatrix}
\begin{pmatrix}
v_1\\
v_2
\end{pmatrix}
\]
\noindent with $0\neq\alpha\in\C$. It is clear then that $\mathcal{V}_2^2$
acts as the identity on $V$, but the subrepesentation $W=\{(v_1,0)\mid v_1\in V(0)\}$
is not a direct summand of $V$ as a $\Anlg{\gl_2}$-module.
\end{example}


\section{A Technique for Creating Galois Orders from Galois Rings Via
Localization}\label{sec: locolization theorem}

In this section, we describe a technique that allows us to turn a Galois ring
into a Galois order involving localization. We use this technique on a toy example
and a localized version of $\Anlg{\gl_n}$ denoted $\tAnlg{\gl_n}$
(see Definition \ref{def: localized A(gln)}).

\subsection{The general result}
We recall that Proposition \ref{prop: Gamma maxl comm in a Galois Order}
states that $\Gamma$ is maximal commutative in a Galois $\Gamma$-order. We observe that
for a general Galois $\Gamma$-ring $\mathscr{U}$, while $\Gamma$ might not be maximal
commutative, its centralizer $C_{\mathscr{U}}(\Gamma)$ in $\mathscr{U}$ will be \cite{FO10}.
This can be seen from the following remark:
\begin{remark}
For Galois $\Gamma$-ring $\mathscr{U}$, the centralizer of $\Gamma$ in $\mathscr{U}$,
denoted $C_{\mathscr{U}}(\Gamma)$, is equal to $\mathscr{U}\cap K$.
\end{remark}

First we define a subring of $L$ that is needed in our result.
\begin{definition}\label{def: ring of coefficients}
Let $\mathscr{U}$ be a subalgebra of $\mathscr{L}$. We define the
\emph{ring of coefficients} of $\mathscr{U}$:
\[
D_\mathscr{U}:=\langle\alpha\in L\mid \exists X\in\mathscr{U}\text{ such that $\alpha$
is a left coefficient of some $\mu\in\supp_\mathscr{M}X$}\rangle_{\rm ring}.
\]
Similarly, we define the \emph{opposite ring of coefficients} of $\mathscr{U}$, denoted
$D_{\mathscr{U}}^{\rm op}$, using right coefficients.
\end{definition}

Now for the result.

\begin{theorem}\label{thm: Order if centralizer is a nice localization for arbitrary G}
Let $G$ be arbitrary and $\mathscr{U}$ be a Galois $\Gamma$-ring in $(L\#\mathscr{M})^G$.
If $C=C_{\mathscr{U}}(\Gamma)$ is the $G$ invariants of the localization of $\Lambda$
with respect to a set that is $\mathscr{M}$-invariant, that is $C=(S^{-1}\Lambda)^G$,
where $S$ is $\mathscr{M}$-invariant, and $D_\mathscr{U}$ is a finitely generated
module over $C$, then $\mathscr{U}$ is a Galois $C$-order in $(L\#\mathscr{M})^G$.
Moreover, if $D_\mathscr{U}\subseteq S^{-1}\Lambda$ (resp. $D_{\mathscr{U}}^{\rm op}\subset S^{-1}\Lambda$),
then $\mathscr{U}$ is a (co-)principal Galois $C$-order.
\end{theorem}
\begin{proof}
First, we find a $\Lambda^\prime$ such that $(\Lambda^\prime,G,\mathscr{M})$
satisfies the assumptions in Section \ref{sec: Galois Orders}. We define
$\Lambda^\prime=\vl{C}$, the integral closure of $C$ in $L$. We observe that
$C=(S^G)^{-1}\Gamma$. As such, $C$ is a localization, and it follows that:
\begin{equation}
\vl{C}=(S^G)^{-1}\vl{\Gamma}=S^{-1}\Lambda.
\end{equation}
Since $S$ is $\mathscr{M}$-invariant and $\vl{C}$ is integral over $C$,
it follows that $\mathscr{M}$ and $G$ are subgroups of $\Aut(\Lambda^\prime)$.
The first two assumptions clearly hold, and the third follows by
$\Lambda^\prime=S^{-1}\Lambda$.

We have that $\mathscr{U}$ is a Galois $C$-ring since it is a Galois
$\Gamma$-ring and $\Frac(C)=\Frac(\Gamma)=K$. All that remains is to show
that $\mathscr{U}$ is a Galois $C$-order. We consider $W\subset\mathscr{L}$
a finite-dimensional left $L$-subspace and aim to show that $W\cap\mathscr{U}$
is finitely generated as a left $C$-module. $W$ has a finite basis
$w_1,\ldots,w_n$ such that:
\[
W=\{\sum \alpha_iw_i\mid \alpha_i\in L\}.
\]
Note that for each $i$, $w_i=\sum_{\mu\in\mathscr{M}} \beta_{i,\mu} \mu$;
as such, since $C$ is a localization of a Noetherian ring and therefore Noetherian,
WLOG we can assume $w_i=\mu_i$ for some $\mu_i\in\mathscr{M}$.
Hence:
\[
W=\sum_i L\mu_i.
\]
So, $W\cap\mathscr{U}\subset\sum_i D_\mathscr{U}\mu_i$, and is
therefore finitely generated. A similar argument justifies the claim if $W$
is instead a right $L$-module. Therefore, $U$ is a Galois $C$-order. 

If additionally we assume $D_\mathscr{U}\subset S^{-1}\Lambda$, we need
to show that $X(c)\in C$ for all $X\in\mathscr{U}$ and $c\in C$. So, we consider
an arbitrary $c\in C$ and $X\in\mathscr{U}$. By Lemma 2.19 in \cite{Hartwig20},
it follows that $X(c)\in K$. Since $C=(S^G)^{-1}\Gamma$, it follows that
$X(c)\in S^{-1}\Lambda$. As such:
\begin{equation}
X(c)\in S^{-1}\Lambda\cap K=(S^{-1}\Lambda)^G=C.
\end{equation}
Thus $X(c)\in C$. If instead $D_{\mathscr{U}}^{\rm op}\subset S^{-1}\Lambda$,
a similar argument shows that $X^\dagger(c)\in C$, thereby proving the claim.
\end{proof}

The above theorem also gives an alternate proof to one direction of Corollary 2.15 in \cite{Hartwig20}.

\subsection{A toy example}\label{subsec: Toy example}
In this subsection, we provide a family of simple examples of Galois rings
to which Theorem \ref{thm: Order if centralizer is a nice localization for arbitrary G}
can be applied.

Let $\Lambda=\C[x]$ the polynomial algebra in one indeterminate $x$, $\delta\in\Aut{}{\Lambda}$ such that $\delta(x)=x-1$,
$\mathscr{M}=\langle\delta\rangle_{\rm grp}$, and $G$ the trivial group.
Then, let $\mathscr{L}=L\#\mathscr{M}$ be the skew-monoid ring and $f(x)\in\C[x]$
such that $f(0)\neq0$. We define $X,Y\in\mathscr{L}$ such that:
\begin{equation}
X:=\delta\frac{f(x)}{x}\qquad\text{and}\qquad Y:=\delta^{-1}.
\end{equation}
Let $U_f=\C\langle\Lambda,X,Y\rangle_{\rm alg}$ and $C_{U_f}(\Lambda)(=C_{U_f})$
the centralizer of $\Lambda$ in $U_f$. We note, as $G$ is trivial,
that $\Lambda=\Gamma$. First, we will show that $U_f$ is Galois $\Gamma$-ring.

\begin{proposition}
The algebra $U_f$ is a Galois $\Gamma$-ring in $L\#\mathscr{M}$.
\end{proposition}
\begin{proof}
This immediately follows from Proposition \ref{prop: Gamma Ring Alt Conditions}
letting $\mathscr{X}=\{X,Y\}$.
\end{proof}

In order to apply Theorem \ref{thm: Order if centralizer is a nice localization for arbitrary G},
we need to describe $C_{U_f}$. The next three lemmas are used to do just that.

\begin{lemma}\label{lem: k=0,-1}
For any $f(x)$ such that $f(0)\neq0$, we have $\dl\frac{1}{x},\frac{1}{x-1}\in C_{U_f}$.
\end{lemma}
\begin{proof}
First, we show that $\dl\frac{1}{x}\in C_{U_f}$. Now, $f(x)=a_nx^n+\cdots+a_1x+a_0$ with $a_0\neq0$
by assumption. As such:
\begin{align*}
\frac{f(x)}{x}&=a_nx^{n-1}+a_{n-1}x^{n-2}+\cdots a_1 +\frac{a_0}{x}\\
&\Rightarrow\frac{1}{x}=a_0^{-1}\bigg(\frac{f(x)}{x}-(a_nx^{n-1}+a_{n-1}x^{n-2}+\cdots a_1)\bigg).
\end{align*}
This shows that $\dl\frac{1}{x}\in C_{U_f}$. To see that $\dl\frac{1}{x-1}\in C_{U_f}$, we follow a similar
division algorithm argument with $\dl\frac{f(x-1)}{x-1}$.
\end{proof}

\begin{lemma}\label{lem: x+k}
For any $f(x)$ such that $f(0)\neq0$ and $k\geq1$, we have $\dl\frac{1}{x+k}\in C_{U_f}$.
\end{lemma}
\begin{proof}
Let $m$ be the order of $(x+k)$ in $\dl\prod_{j=0}^{k-1}f(x+j)$. Then consider the following:
\begin{align*}
Y^{k+1}(XY)^{m}X^{k+1}&=\delta^{-k-1}\bigg(\frac{f(x-1)}{x-1}\bigg)^{m}\delta^{k+1}\prod_{j=0}^{k}\frac{f(x+j)}{x+j}\\
&=\bigg(\frac{f(x+k)}{x+k}\bigg)^{m}\prod_{j=0}^k\frac{f(x+j)}{x+j}\\
&=\bigg(\frac{f(x+j)}{x+j}\bigg)^{m+1}\prod_{j=0}^{k-1}\frac{f(x+j)}{x+j}
\end{align*}
As such, there are $m$ factors of $(x+k)$ in the numerator and $m+1$ factors in the
denominator. Thus, multiplying by $\dl\prod_{j=0}^{k-1}(x+j)$ and using
a division algorithm argument, it follows that $\dl\frac{1}{x+k}\in C_{U_f}$.
\end{proof}

\begin{lemma}\label{lem: x-k}
For any $f(x)$ such that $f(0)\neq0$ and $k\geq2$, we have $\dl\frac{1}{x-k}\in C_{U_f}$.
\end{lemma}
\begin{proof}
Let $m$ be the order of $(x-k)$ in $\dl\prod_{j=1}^{k-1}f(x-j)$. Then consider
the following:
\begin{align*}
X^{k}(YX)^{m}Y^{k}&=\delta^{k}\prod_{j=0}^{k-1}\frac{f(x+j)}{x+j}\bigg(\frac{f(x)}{x}\bigg)^{m}\delta^{-k}\\
&=\prod_{j=0}^{k-1}\frac{f(x+j-k)}{x+j-k}\bigg(\frac{f(x-k)}{x-k}\bigg)^{m}\\
&=\bigg(\frac{f(x-k)}{x-k}\bigg)^{m+1}\prod_{\ell=1}^{k-1}\frac{f(x-\ell)}{x-\ell}.
\end{align*}
As such, there are $m$ factors of $(x-k)$ in the numerator and $m+1$ factors in the
denominator. Thus, multiplying by $\dl\prod_{j=1}^{k-1}(x-j)$ and using a
division algorithm argument, it follows that $\dl\frac{1}{x-k}\in C_{U_f}$.
\end{proof}
\begin{proposition}\label{prop: description of centralizer in toy example}
If $f(x)$ is a polynomial such that $f(0)\neq0$, then
$C_{U_f}=\C[x]\bigg[\frac{1}{x+k}~\bigg\vert~k\in\Z\bigg]$.
\end{proposition}
\begin{proof}
$C_{U_f}\supseteq\C[x]\bigg[\frac{1}{x+k}~\bigg\vert~k\in\Z\bigg]$ by
Lemmas \ref{lem: k=0,-1}, \ref{lem: x+k}, and \ref{lem: x-k}.
To show the reverse inclusion, we observe that for $Z\in C_{U_f}$, $Z$ must be
of "degree 0" with regards to $\delta$ that is:
\begin{align*}
Z&=\sum_{k=1}^m g_k(x) \prod_{n=0}^{\infty}(X^nY^n)^{k_{-n}}(Y^nX^n)^{k_n}\\
&=\sum_{k=1}^m g_k(x)\prod_{\ell=-\infty}^{\infty}
\left(\frac{f(x+\ell)}{(x+\ell)}\right)^{k_\ell}\\
&=\sum_{k=1}^m G_k(x)\prod_{\ell=-\infty}^{\infty}\frac{1}{(x+\ell)^{k_\ell}}\in\C[x]\bigg[\frac{1}{x+k}~\bigg\vert~k\in\Z\bigg],
\end{align*}
where $k_\ell\neq0$ for at most finitely many terms. Thus
$C_{U_f}\subseteq\C[x]\bigg[\frac{1}{x+k}~\bigg\vert~k\in\Z\bigg]$.
\end{proof}

We can now prove that $U_f$ is a Galois $C_{U_f}$-order using
Theorem \ref{thm: Order if centralizer is a nice localization for arbitrary G}.

\begin{corollary}\label{cor: toy example is an order}
The algebra $U_f$ is a principal and co-principal Galois $C_{U_f}$-order
in $L\#\mathscr{M}$.
\end{corollary}
\begin{proof}
Proposition \ref{prop: description of centralizer in toy example} gives us that
the main supposition of Theorem \ref{thm: Order if centralizer is a nice localization for arbitrary G}.
All that remains to show is $D_{U_f},D_{U_f}^{\rm op}\subset S^{-1}\Lambda=C_{U_f}$ in this case.
However, this is clear since $U_f$ is generated by $X$, $Y$, and $\Lambda$.
\end{proof}

\subsection{Localizing \texorpdfstring{$\Anlg{\gl_n}$}{A(gln)}}\label{subsec: localizing A(gln)}

In this subsection, we construct a localization of $\Anlg{\gl_n}$ denoted
$\tAnlg{\gl_n}$, to which Theorem \ref{thm: Order if centralizer is a nice localization for arbitrary G}.
can be applied.

In order to construct this localization, we describe shifted Vandermonde
polynomials using the following notation:

\begin{notation}
Let $\mathcal{V}_k$ be the Vandermonde in the $x_{ki}$ variables. Let
$l:=(l_1,l_2,\ldots,l_{k-1})\in\Z^{k-1}$. We denote the ($l$-)shifted $\mathcal{V}_k$ as follows:
\[
\mathcal{V}_{k,l}:=\prod_{i<j}(x_{ki}-x_{kj}+\sum_{n=i}^{j-1}l_n).
\]
This notation makes sense because for $i<j$:
\[
x_{ki}-x_{kj}=(x_{ki}-x_{k,i+1})+(x_{k,i+1}-x_{k,i+2})+\cdots+(x_{k,j-1}-x_{kj}).
\]
Therefore, any shift of $\mathcal{V}_k$ is uniquely determined by the shifts of $x_{ki}-x_{k,i+1}$ for
$i=1,2,\ldots,k-1$.
\end{notation}

Now to construct our localization.

\begin{definition}\label{def: localization set for A(gln)}
Let $S:=\langle\mathcal{V}_{k,l}\mid l\in\Z^{k-1};~ k=2,\ldots,n-1\rangle_{\rm monoid}$.
We observe that $S$ is a multiplicatively closed set in $\Lambda$, and
$\Anlg{\gl_n}\subset (S^{-1}\Lambda\#\mathscr{M})^{\mathbb{A}_n}$. We also note
that $S$ is the smallest $\mathscr{M}$-invariant multiplicatively closed set
that contains $\mathcal{V}_2,\ldots,\mathcal{V}_{n-1}$.
\end{definition}

As Example \ref{ex: Gammatilde is not maximal comm} demonstrates,
$C_{\Anlg{\gl_n}}(\widetilde{\Gamma})\subset(S^{-1}\Lambda)^{\mathbb{A}_n}$.
It is not known if this containment is strict, so this motivates the construction
of the following localization of $\Anlg{\gl_n}$.

\begin{definition}\label{def: localized A(gln)}
Our new algebra of interest in $\widetilde{\mathscr{K}}$ is
$\tAnlg{\gl_n}:=\C\langle U_n,(S^{-1}\Lambda)^{\mathbb{A}_n}\rangle_{\rm alg}$.
Notice this coincides with the definitions of $\Anlg{\gl_2}$ for $n=2$.
\end{definition}

\begin{remark}
It follows from Lemma 2.10 in \cite{Hartwig20} that $\tAnlg{\gl_n}$
is a Galois $\widetilde{\Gamma}$-ring since it contains $\Anlg{\gl_n}$.
Moreover, $C_{\tAnlg{\gl_n}}(\widetilde{\Gamma})=(S^{-1}\Lambda)^{\mathbb{A}_n}$
as well.
\end{remark}

\begin{remark}
In $\tAnlg{\gl_n}$, relation \ref{rel: weird relation in A(gl3)} from
Section \ref{subsec: partial presentation} can be rewritten either as
\begin{itemize}
\item[\ref{rel: weird relation in A(gl3)}$^\prime$]
$\displaystyle [A_{21}^\pm,A_{22}^\pm]=\frac{\pm2}{\mathcal{V}_2\pm1}A_{21}^\pm A_{22}^\pm$, or\\
\item[\ref{rel: weird relation in A(gl3)}$^{\prime\prime}$] $\displaystyle A_{22}^\pm A_{21}^\pm
=\frac{\mathcal{V}_2\mp1}{\mathcal{V}_2\pm1}A_{21}^\pm A_{22}^\pm$.
\end{itemize}
\end{remark}

\begin{corollary}\label{cor: Localized A(gln) is an order}
The subalgebra $\tAnlg{\gl_n}\subset\widetilde{K}$ is both a principal and co-principal
Galois $(S^{-1}\Lambda)^{\mathbb{A}_n}$-order.
\end{corollary}
\begin{proof}
It is clear by construction that $\tAnlg{\gl_n}$ satisfies the main supposition
of Theorem \ref{thm: Order if centralizer is a nice localization for arbitrary G}.
Also, it follows from the definition of the $a_{ki}^\pm$'s in 
(\ref{eq: defintion of the rational functions a}) that
$D_{\tAnlg{\gl_n}},D_{\tAnlg{\gl_n}}^{\rm op}\subseteq S^{-1}\Lambda$. We can therefore apply
Theorem \ref{thm: Order if centralizer is a nice localization for arbitrary G}.
\end{proof}

In \cite{Webster19}, it was shown that every (co-)principal Galois order has a corresponding (co-)principal
flag order. This leads us to the following:

\begin{openprob}
What is the corresponding (co-)principal flag order of $\tAnlg{\gl_n}$?
\end{openprob}


\section{(Generic) Gelfand-Tsetlin Modules over \texorpdfstring{$\Anlg{\gl_n}$}{A(gln)}}\label{sec: GT Modules over A(gln)}

\subsection{Some general results}

Following the techniques in \cite{early_mazorchuk_vishnyakova_2018} and \cite{Hartwig20},
we construct canonical simple Gelfand-Tsetlin modules over $\tAnlg{\gl_n}$. We need the
following additional assumptions for these next two results:
\[
\begin{array}{lcl}
{\rm (A4)}& &\text{$\Lambda$ is finitely generated over an algebraically closed field $\mathbbm{k}$ of characteristic 0},\\
{\rm (A5)}& &\text{$G$ and $M$ act by $\mathbbm{k}$-algebra homomorphisms on $\Lambda$}.
\end{array}
\]
Let $\hat{\Gamma}$ be the set of all $\Gamma$-characters (i.e., $\mathbbm{k}$-algebra homomorphisms
$\xi\colon\Gamma\rightarrow\mathbbm{k}$).

\begin{definition}
Let $\mathscr{U}$ be a Galois $\Gamma$-ring in $\mathscr{K}$. A left
$\mathscr{U}$-modules $V$ is said to be a \emph{Gelfand-Tsetlin module (with
respect to $\Gamma$)} if $\Gamma$ acts locally finitely on $V$. Equivalently:
\[
V=\bigoplus_{\xi\in\hat{\Gamma}}V_\xi,\qquad V_\xi=\{v\in V\mid(\ker\xi)^Nv=0,N\gg0\}.
\]
Similarly, one can define a right Gelfand-Tsetlin modules.
\end{definition}

The details for the following lemma can be found in \cite{DFO94}.

\begin{lemma}
Let $\mathscr{U}$ be a Galois $\Gamma$-ring in $\mathscr{K}$.
\begin{enumerate}[\rm (i)]
\item Any submodule and any quotient of a Gelfand-Tsetlin module is a
Gelfand-Tsetlin module.
\item Any $\mathscr{U}$-module generated by generalized weight vectors is a
Gelfand-Tsetlin module.
\end{enumerate}
\end{lemma}

\begin{theorem}[\cite{Hartwig20}, Theorem 3.3 (ii)]\label{thm: canonical GZ mods for coprincpical Galois Order}
Let $\xi\in\hat{\Gamma}$ be any character. If $\mathscr{U}$ is a co-principal Galois
$\Gamma$-order in $\mathscr{K}$, then the left cyclic $\mathscr{U}$-module
$\mathscr{U}\xi$ has a unique simple quotient $V^\prime(\xi)$. Moreover,
$V^\prime(\xi)$ is a Gelfand-Tsetlin over $\mathscr{U}$ with $V^\prime(\xi)_\xi\neq0$
and is called the canonical simple left Gelfand-Tsetlin $\mathscr{U}$-module
associated to $\xi$.
\end{theorem}

\subsection{The case of \texorpdfstring{$\Anlg{\gl_n}$}{A(gln)}}

We note that for $n\geq3$ that $\widetilde{\Lambda}$ is not finitely generated as a
$\C$-algebra. This prevents us from using all of the results as is, but all is not
lost. The main arguments of Theorem
\ref{thm: canonical GZ mods for coprincpical Galois Order} rests on:
\[
\Hom{\Gamma}{\Gamma/\mathfrak{m},\Gamma^\ast}\cong\Hom{\mathbbm{k}}{\Gamma
/\mathfrak{m}\otimes_\Gamma\Gamma,\mathbbm{k}}\cong\mathbbm{k}.
\]
If we want a similar result for $S^{-1}\widetilde{\Gamma}$ we need to recall that
every maximal ideal $\mathfrak{m}$ of $S^{-1}\widetilde{\Gamma}$ is of the form
$S^{-1}\mathfrak{p}$, where $\mathfrak{p}$ is a prime (not necessarily maximal)
ideal of $\widetilde{\Gamma}\setminus S$. Therefore we have the following result.

\begin{theorem}\label{thm: Canonical GT modules for localized A(gln)}
Let $\xi$ be a character of $S^{-1}\widetilde{\Gamma}$ such that
$\ker\xi=S^{-1}\mathfrak{m}$, for some maximal ideal $\mathfrak{m}$ of
$\widetilde{\Gamma}$. Then the left cyclic module $\tAnlg{\gl_n}\xi$ has a unique
simple quotient $V^\prime(\xi)$ which is a Gelfand-Tsetlin module over
$\tAnlg{\gl_n}$ with $V^\prime(\xi)_\xi\neq0$.
\end{theorem}
\begin{proof}
The key difference in this proof compared to Theorem
\ref{thm: canonical GZ mods for coprincpical Galois Order} is observing that
\[
S^{-1}\widetilde{\Gamma}/S^{-1}\mathfrak{m}\cong S^{-1}(\widetilde{\Gamma}/\mathfrak{m})
\cong\mathbbm{k}.
\]
Otherwise, the proof follows the same structure.
\end{proof}

Since $\tAnlg{\gl_n}$ is created by localizing $\widetilde{\Gamma}$ and $\Lambda$, we can view
any $\tAnlg{\gl_n}$-module $V$ as a $\Anlg{\gl_n}$-module by precomposing
with the embedding $\iota\colon\Anlg{\gl_n}\hookrightarrow\tAnlg{\gl_n}$.


\section{Gelfand-Kirillov Conjecture for \texorpdfstring{$\Anlg{\gl_n}$}{A(gln)}}\label{sec: GK conj for A(gln)}

In this section we will discuss for which $n$'s the algebras $\Anlg{\gl_n}$ and
$\tAnlg{\gl_n}$ satisfy the Gelfand-Kirillov Conjecture. This is related to the
Noncommutative Noether Problem for the alternating group $A_n$, as discussed in
\cite{FSNoncommutativeNoetherVSClassicalNoether}.

An algebra $A$ is said to satisfy Gelfand-Kirillov Conjecture if it
is birationally equivalent to a Weyl algebra. That is its skew-field of fractions
is isomorphic to a skew Weyl field.

\begin{lemma}
$\Frac(\tAnlg{\gl_n})=\Frac(\Anlg{\gl_n})$.
\end{lemma}
\begin{proof}
This follows because $\tAnlg{\gl_n}$ is created by localizing $\widetilde{\Gamma}$ and
$\Lambda$.
\end{proof}

Hence, $\tAnlg{\gl_n}$ and $\Anlg{\gl_n}$ either both will or will not satisfy
the Gelfand-Kirillov Conjecture for each $n$.

\begin{proposition}\label{prop: frac of A(gln)}
For every $n$,
\[
\Frac(\Anlg{\gl_n})\cong\Frac\Big(\C(x_1,\ldots,x_n)^{A_n}\otimes
\bigotimes_{k=1}^{n-1}\big(\Frac(\mathscr{W}_k(\C))\big)^{A_k}\Big),
\]
where $\mathscr{W}_k(\C)$
is the $k$-dimensional Weyl algebra over $\C$.
\end{proposition}
\begin{proof}
It is clear by construction that:
\begin{equation}\label{eq: begining chain in frac of A(gln) proof}
\Frac(\Anlg{\gl_n})=\Frac(\mathscr{L}^{\mathbb{A}_n})
=\Frac((L\#\mathscr{M})^{\mathbb{A}_n}).
\end{equation}
Since $L=\Frac(\Lambda)$:
\begin{equation}
\Frac((L\#\mathscr{M})^{\mathbb{A}_n})\cong\Frac((\Lambda\#\mathscr{M})^{\mathbb{A}_n}).
\end{equation}
We now recall that $\mathscr{M}$ is generated by $\delta^{ki}$'s and $\delta^{ki}$ fixes
$x_{\ell j}$ if $\ell\neq k$. As such, we have:
\begin{equation}
\Frac((\Lambda\#\mathscr{M})^{\mathbb{A}_n})\cong
\Frac((\Lambda_n\otimes\bigotimes_{k=1}^{n-1}\Lambda_k\#\mathscr{M}_k)^{\mathbb{A}_n}),
\end{equation}
where $\Lambda_k=\C[x_{k1},\ldots,x_{kk}]\subset\Lambda$ and $\mathscr{M}_k
=\langle\delta^{ki}\mid 1\leq i\leq k\rangle_{\rm grp}\leq\mathscr{M}$. Now, the $k$-th component
of $\mathbb{A}_n$ acts only on the $k$-th component of the tensor product. Therefore:
\begin{equation}
\Frac((\Lambda_n\otimes\bigotimes_{k=1}^{n-1}\Lambda_k\#\mathscr{M}_k)^{\mathbb{A}_n})\cong
\Frac(\Lambda_n^{A_n}\otimes\bigotimes_{k=1}^{n-1}(\Lambda_k\#\mathscr{M}_k)^{A_k}).
\end{equation}
Finally, since $A_k$ is finite for each $k$ we have:
\begin{equation}\label{eq: ending chain in frac of A(gln) proof}
\Frac(\Lambda_n^{A_n}\otimes\bigotimes_{k=1}^{n-1}(\Lambda_k\#\mathscr{M}_k)^{A_k})\cong
\Frac\big((\Frac(\Lambda_n))^{A_n}\otimes\bigotimes_{k=1}^{n-1}
(\Frac(\Lambda_k\#\mathscr{M}_k))^{A_k}\big).
\end{equation}
Combining the equations
(\ref{eq: begining chain in frac of A(gln) proof})-(\ref{eq: ending chain in frac of A(gln) proof}),
we have:
\begin{equation}\label{eq: last equation in Frac A(gln) proof}
\Frac(\Anlg{\gl_n})\cong
\Frac\big((\Frac(\Lambda_n))^{A_n}\otimes\bigotimes_{k=1}^{n-1}
(\Frac(\Lambda_k\#\mathscr{M}_k))^{A_k}\big).
\end{equation}
We finish the proof by observing that $\Frac(\Lambda_n)\cong\C(x_1,\ldots,x_n)$
and $\Lambda_k\#\mathscr{M}_k\cong\mathscr{W}_k(\C)$ by sending $\delta^{ki}x_{ki}\mapsto X_i$ and
$(\delta^{ki})^{-1}\mapsto Y_i$.
\end{proof}

We recall for readers both the classical Noether's problem and the
noncommutative Noether's problem as stated in \cite{FSNoncommutativeNoetherVSClassicalNoether}.
The classical problem asks, given a finite group $G$ and a rational function field
$\mathbbm{k}(x_1,\ldots,x_n)$ over a field $\mathbbm{k}$ such that $G$ acts linearly
on $\mathbbm{k}(x_1,\ldots,x_n)$, is $\mathbbm{k}(x_1,\ldots,x_n)^G$ a purely transcendental
extension of $\mathbbm{k}$. The noncommutative problem exchanges the rational function field
with the skew field of fractions of a Weyl algebra and asks if the $G$ invariants are the skew
field of some purely transcendental extension of $\mathbbm{k}$.

\begin{theorem}[Theorem 1.1 in \cite{FSNoncommutativeNoetherVSClassicalNoether}]\label{thm: CNP implies NNP}
If $G$ satisfies the Commutative Noether's problem, then $G$ satisfies the Noncommutative
Noether's Problem.
\end{theorem}

Noether's problem for $A_n$ is still open for $n\geq5$. However, we obtain the following
result:

\begin{theorem}\label{thm: Noethers problem implies GK-conjecture}
If the alternating groups $A_1,A_2,\ldots,A_n$ provide a positive solution to Noether's problem, then
$\Anlg{\gl_n}$ satisfies the Gelfand-Kirillov conjecture.
\end{theorem}
\begin{proof}
If $A_k$ satisfies Noether's problem, then
$\Frac(\mathscr{W}_k(\C))^{A_k}\cong\Frac(\mathscr{W}_k(\C))$. The rest follows from
Proposition \ref{prop: frac of A(gln)}.
\end{proof}

Hence, as a corollary to Theorem \ref{thm: Noethers problem implies GK-conjecture}
and Maeda's results in \cite{MAEDA1989418}, we have:
\begin{corollary}
For $n\leq 5$, $\Anlg{\gl_n}$ satisfies the Gelfand-Kirillov Conjecture.
\end{corollary}


\section*{Acknowledgements}
The author would like to thank his advisor Jonas Hartwig for guidance and helpful
discussion. The author would also like to thank Jo\~ao Schwarz for his comments and helpful discussion. Finally, the author would like to thank Iowa State University, where the author resided during the completion of this work.

\end{document}